\documentclass{article}
\usepackage{latexsym,amsfonts,amsthm,amsmath,amscd,amssymb}
\usepackage[dvips]{graphicx}

\newtheorem{lemma}{Lemma}[section]
\newtheorem{thm}[lemma]{Theorem}
\newtheorem{rem}[lemma]{Remark}
\newtheorem{prop}[lemma]{Proposition}
\newtheorem{cor}[lemma]{Corollary}

\newtheorem{oss}[lemma]{Observation}
\newtheorem{example}[lemma]{Example}
\newtheorem{defn}[lemma]{Definition}

\newcommand\matR{{\mathbb{R}}}
\newcommand\matN{{\mathbb{N}}}

\newcommand\calD{{\mathcal D}}

\begin{document}

\title{Diffeological connections on diffeological vector pseudo-bundles}

\author{Ekaterina~{\textsc Pervova}}

\maketitle

\begin{abstract}
\noindent We consider one possible definition of a diffeological connection on a diffeological vector pseudo-bundle. It is different from the one proposed in [7] and is in fact simpler, since it is obtained
by a straightforward adaption of the standard definition of a connection as an operator on the space of all smooth sections. One aspect prominent in the diffeological context has to do with the choice of
an appropriate substitute for tangent vectors and smooth vector fields, since there are not yet standard counterparts for these notions. In this respect we opt for the simplest possibility; since there is an
established notion of the (pseudo-)bundle of differential forms on a diffeological space, we take the corresponding dual pseudo-bundle to play the role of the tangent bundle. Smooth vector fields are
then smooth sections of this dual pseudo-bundle; this is one reason why we devote a particular attention to the space of smooth sections of an arbitrary diffeological vector pseudo-bundle (one curiosity
is that it might easily turn out to be infinite-dimensional, even when the pseudo-bundle itself has a trivial finite-dimensional vector bundle as the underlying map). We concentrate a lot on how this space
interacts with the gluing construction for diffeological vector pseudo-bundles (described in [10]). We then deal with the same question for the proposed notion of a diffeological connection.

\end{abstract}

\section*{Introduction}

Diffeology can be seen as a way to extend the field of application of differential geometry (or of differential calculus, according to some). There have been, and are, other attempts to do this; some of
these approaches are summarized in \cite{St}. Diffeological spaces first appeared in \cite{So1}, \cite{So2}; a lot of fundamental concepts, such as the underlying topology, called \emph{D-topology},
and the counterpart of the fibre bundle, among others, were developed in \cite{iglFibre}. A recent and comprehensive source on the field of diffeology is \cite{iglesiasBook}.

From a certain (necessarily simplistic, but still interesting) point of view, diffeology can be seen as a way to consider any given function as a smooth one --- and then see what happens. This is essentially
the notion of a \emph{diffeology generated by a given plot}; what becomes for instance of the usual $\matR$ if we consider the modulus $|x|$ as a smooth function \emph{into} it? One immediate answer
(there would be of course more intricate ones) is that no linear function on it is smooth then (except the zero one); and this is just the most basic of examples. This is the kind of a straightforward (it can be
said, naive) approach that we opt for in this paper.

\paragraph{The notion of a connection} A certain preliminary notion of a diffeological connection is sketched out in \cite{iglesiasBook}. Our approach is different from one therein, but it is very much 
straightforward. A usual connection on a smooth vector bundle $E\to M$ over a smooth manifold $M$ can be defined as a smooth operator $C^{\infty}(M,E)\to C^{\infty}(M,T^*M\otimes E)$, that is linear 
and obeys the Leibniz rule. For all objects that appear in its definition, there are well-established diffeological counterparts, with the bundle of (values of) \emph{diffeological differential $1$-forms} 
$\Lambda^1(X)$ over a diffeological space $X$ (see \cite{iglesiasBook} again, although it is not the original source) taking the place of the cotangent bundle. Thus, the definition-by-analogy of a diffeological 
connection on a \emph{diffeological vector pseudo-bundle} $\pi:V\to X$ is an obvious matter; it suffices to substitute $X$ for $M$, $V$ for $E$, and consider diffeological forms instead of sections of the 
cotangent bundle. A few minor details need to be explained (which we do), and it also should be specified that the covariant derivatives are taken with respect to sections of the dual pseudo-bundle 
$(\Lambda^1(X))^*$, which for us plays the role of the tangent bundle (of which there is not yet a standard theory in diffeology). However, covariant derivatives is the only place where we need tangent 
vectors.

Most of what we do is devoted to constructing connections on pseudo-bundles obtained by \emph{diffeological gluing} (see \cite{pseudobundles}). To this end we first dedicate significant attention 
to the behavior of the spaces of sections under gluing. Thus, if $\pi_1:V_1\to X_1$ and $\pi_2:V_2\to X_2$ are two pseudo-bundles, and $\pi_1\cup_{(\tilde{f},f)}\pi_2:V_1\cup_{\tilde{f}}V_2\to X_1\cup_f X_2$ 
is the result of their gluing (see below for the precise definition), the space of sections $C^{\infty}(X_1\cup_f X_2,V_1\cup_{\tilde{f}}V_2)$ is a smooth surjective image of a subset of the direct product 
$C^{\infty}(X_1,V_1)\times C^{\infty}(X_2,V_2)$ (Section 2). We use this to show that if $V_1$ and $V_2$ are both endowed with connections, and these connections satisfy a specified \emph{compatibility 
condition}, then there is an induced connection on $V_1\cup_{\tilde{f}}V_2$. If, finally, $V_1$ and $V_2$ are endowed with \emph{pseudo-metrics} $g_1$ and $g_2$ (these are diffeological counterparts of 
Riemannian metrics) that are well-behaved with respect to each other, and the two connections on $V_1$ and $V_2$ are compatible with these pseudo-metrics, then they, again, induce a connection on 
$V_1\cup_{\tilde{f}}V_2$; this resulting connection is compatible with $\tilde{g}$, a pseudo-metric determined by $g_1$ and $g_2$.

\paragraph{Diffeological gluing} A large part of our approach consists in establishing how the above-listed components behave with respect to the operation of \emph{diffeological gluing}. On the level
of underlying sets, this is the standard operation of topological gluing; the resulting space is endowed with a diffeology that is probably the finest sensible one: it naturally includes the diffeologies on the
factors, and not much else. One disadvantage of this notion is that this is a pretty weak diffeology, that loses (or risks losing) sight of some natural aspects of the underlying space; for instance, the
obvious gluing diffeology on the union of the coordinate axes in $\matR^2$ is weaker than the subset diffeology relative to its inclusion into $\matR^2$, see \cite{watts} (on the other hand, the concept
of gluing may provide a natural framework for treating objects such as manifolds with corners, see below for more detail). As of now, we view this notion of the gluing diffeology as more of a precursor
to a coarser one, with more involved properties --- but still as a useful testing ground for the constructions that we are considering.

\paragraph{Acknowledgements} Without any trace of doubt, my gratitude goes to Prof. Riccardo Zucchi for providing a precious and, most of all, consistent support which allowed this paper to be
completed.

\section{The main notions}

Here we briefly recall the notions of diffeology that appear throughout this paper, in the form in which they appear in \cite{iglesiasBook}.

\subsection{Diffeological spaces and vector spaces}

The central object for diffeology is a set $X$ endowed with a \emph{diffeological structure}, which is a collection of maps from usual domains to $X$; three natural conditions must be satisfied.

\begin{defn} \emph{(\cite{So2})} A \textbf{diffeological space} is a pair $(X,\calD_X)$ where $X$ is a set and $\calD_X$ is a specified collection, also called the \textbf{diffeology} of $X$ or its
\textbf{diffeological structure}, of maps $U\to X$ (called \textbf{plots}) for each open set $U$ in $\matR^n$ and for each $n\in\matN$, such that for all open subsets $U\subseteq\matR^n$ and
$V\subseteq\matR^m$ the following three conditions are satisfied:
\begin{enumerate}
\item (The covering condition) Every constant map $U\to X$ is a plot;
\item (The smooth compatibility condition) If $U\to X$ is a plot and $V\to U$ is a smooth map (in the usual sense) then the composition $V\to U\to X$ is also a plot;
\item (The sheaf condition) If $U=\cup_iU_i$ is an open cover and $U\to X$ is a set map such that each restriction $U_i\to X$ is a plot then the entire map $U\to X$ is a plot as well.
\end{enumerate}
\end{defn}

A standard example of a diffeological space is a standard manifold whose diffeology consists of all usual smooth maps into it, but many others, and quite exotic, examples can be found. A map $f:X\to Y$ 
between diffeological spaces $X$ and $Y$ is \textbf{smooth} if for every plot $p$ of $X$ the composition $f\circ p$ is a plot of $Y$. Suppose now that only $X$ is endowed with a diffeology; then $Y$ can be 
endowed with the so-called \textbf{pushforward diffeology} with respect to $f$, which is the minimal diffeology for which $f$ is smooth (the map $f$ is then called a \textbf{subduction}). 

Every subset $Y\subseteq X$ of a diffeological space $X$ carries a natural diffeology, called \textbf{the subset diffeology}, which consists of all plots of $X$ whose image is wholly contained in $Y$. Likewise,
the quotient of $X$ by any equivalence relation $\sim$ carries the \textbf{quotient diffeology}, defined as the pushforward of the diffeology of $X$ by the natural projection $X\to X/\sim$.

The \textbf{disjoint union diffeology} on the disjoint union of diffeological spaces $X_1,\ldots,X_n$ is the smallest diffeology such that for each $i=1,\ldots,n$ the natural injection $X_i\to X$; is smooth; the 
\textbf{product diffeology} on their direct product is the coarsest diffeology such that for each $i=1,\ldots,n$ the natural projection $\pi_i:X=X_1\times\ldots\times X_n\to X_i$ is smooth. If $X$ and $Y$ are two 
diffeological spaces, $C^{\infty}(X,Y)$ stands for the set of all smooth maps $X\to Y$ and is endowed with a natural diffeology called the \textbf{functional diffeology}. It is defined as the largest diffeology such 
that the \emph{evaluation map}, $\mbox{\textsc{ev}}:C^{\infty}(X,Y)\times X\to Y$, given by $\mbox{\textsc{ev}}(f,x)=f(x)$, is smooth.

A \textbf{diffeological vector space} (over $\matR$) is a vector space $V$ endowed with a \textbf{vector space diffeology}, that is, any diffeology for which the following two maps are smooth: the addition map 
$V\times V\to V$, where $V\times V$ carries the product diffeology, and the scalar multiplication map $\matR\times V\to V$, where $\matR$ has the standard diffeology and $\matR\times V$ carries the product 
diffeology. Any usual vector subspace of a diffeological vector space is naturally a diffeological vector space for the subset diffeology. The same is true for any quotient of a vector space, which is automatically 
assumed to carry the quotient diffeology. The \textbf{diffeological dual} $V^*$ of a diffeological vector space $V$ (\cite{vincent}, \cite{wu}) is the space of all smooth linear maps with values into the standard 
$\matR$, endowed with the corresponding functional diffeology. 

A scalar product on a diffeological vector space is a smooth non-degenerate symmetric bilinear form $V\times V\to\matR$ (for the product diffeology on $V\times V$ and the standard diffeology on $\matR$). 
However, a scalar product in this sense rarely exists on diffeological vector spaces; in particular, among the finite-dimensional ones, scalar products exist only on those diffeomorphic to the standard $\matR^n$ 
(see \cite{iglesiasBook}). In general, the maximal rank of a smooth symmetric bilinear form on $V$ is equal to the dimension if its diffeological dual $V^*$; a smooth symmetric semidefinite positive bilinear form 
that achieves this rank (there is always one) is called a \textbf{pseudo-metric} on $V$.

The direct sum of diffeological vector spaces is endowed with the \emph{product} diffeology. Given a finite collection $V_1,\ldots,V_n$ of diffeological vector spaces, their usual tensor product 
$V_1\otimes\ldots\otimes V_n$ is endowed with the \textbf{tensor product diffeology} (see \cite{vincent}, \cite{wu}). The tensor product diffeology is defined as the quotient diffeology corresponding to the usual 
representation of $V_1\otimes\ldots\otimes V_n$ as the quotient of the free product $V_1\times\ldots\times V_n$ (that is endowed with the smallest vector space diffeology on the free product containing the 
product, \emph{i.e.} the direct sum, diffeology on $V_1\oplus\ldots\oplus V_n$) by the kernel of the universal map onto $V_1\otimes\ldots\otimes V_n$.

\subsection{Diffeological vector pseudo-bundles and pseudo-metrics on them}

The notion of a diffeological vector pseudo-bundle appeared initially in \cite{iglFibre} as a partial instance of \emph{diffeological fibre bundle} (see also \cite{iglesiasBook}, Chapter 8), then in \cite{vincent}
under the name of a \emph{regular vector bundle}, and finally in \cite{CWtangent} under the name of a \emph{diffeological vector space over $X$}. We use the term \emph{diffeological vector
pseudo-bundle}, in order to emphasize that frequently it is not really a bundle (it is not required to be locally trivial), and also to avoid confusion with individual diffeological vector spaces, something
which might happen with the term adopted in \cite{CWtangent} when both concepts appear simultaneously.

\begin{defn}
A \textbf{diffeological vector pseudo-bundle} is a smooth surjective map $\pi:V\to X$ between two diffeological spaces $V$ and $X$ such that for each $x\in X$ the pre-image $\pi^{-1}(x)$ carries a vector space 
structure, and the corresponding addition and scalar multiplication maps, as well as the zero section $s_0:X\to V$, are smooth for the appropriate diffeologies, that is, the addition map $V\times_X V\to V$ 
is smooth for the subset diffeology on $V\times_X V$ as a subset of $V\times V$, which itself is considered with the product diffeology, the scalar multiplication map $\matR\times V\to V$ is smooth for the 
standard diffeology on $\matR$ and the corresponding product diffeology on $\matR\times V$.
\end{defn}

All usual operations on smooth vector bundles (direct sum, tensor product, and taking duals) are defined for diffeological vector pseudo-bundles as well (see \cite{vincent}, \cite{pseudobundles}). In particular, 
the diffeology on the dual pseudo-bundle $\pi:V\to X$ is described by the following condition: a map $p:\matR^l\supset U\to V^*$ is a plot for the dual bundle diffeology on $V^*$ if and only if for every plot 
$q:\matR^{l'}\supset U'\to V$ of $V$ the map $Y'\to\matR$ defined on $Y'=\{(u,u')|\pi^*(p(u))=\pi(q(u'))\in X\}\subset U\times U'$ and acting by $(u,u')\mapsto p(u)(q(u'))$, is smooth for the subset diffeology of 
$Y'\subset\matR^{l+l'}$ and the standard diffeology of $\matR$. The corresponding subset diffeology on each fibre $V_x^*$ is its usual functional diffeology as the dual space of the diffeological vector space 
$V_x$.

A \textbf{pseudo-metric} on a diffeological vector pseudo-bundle $\pi:V\to X$ is a smooth section of the pseudo-bundle $\pi^*\otimes\pi^*:V^*\otimes V^*\to X$, \emph{i.e.}, a smooth map
$g:X\to V^*\otimes V^*$ such that for all $x\in X$ the value $g(x)$ is a symmetric form of rank $\dim((\pi^{-1}(x))^*)$, with all the eigenvalues non-negative; in other words, it is a pseudo-metric on
the diffeological vector space $\pi^{-1}(x)$. Such a pseudo-metric obviously exists on any trivial diffeological pseudo-bundle, but in general its existence is not guaranteed.

\subsection{Diffeological gluing}

This operation, considered in some detail in \cite{pseudobundles}, mimics the usual topological gluing, with which it coincides on the level of underlying topological spaces. The space comes with the 
standard choice of diffeology, called the \textbf{gluing diffeology}. As we mentioned in the Introduction, this is in some sense the finest diffeology that it makes sense to consider.

Let $X_1$ and $X_2$ be two diffeological spaces, and let $f:X_1\supset Y\to X_2$ be a map smooth for the subset diffeology of $Y$. Then there is a usual topological gluing of $X_1$ to $X_2$ along $f$ 
(symmetric if $f$ is injective), defined as
$$X_1\cup_f X_2=(X_1\sqcup X_2)/\sim,\mbox{ where }X_1\supset Y\ni y\sim f(y)\in X_2.$$ The space $X_1\cup_f X_2$ is endowed with the quotient diffeology of the disjoint union diffeology on 
$X_1\sqcup X_2$ and is said to be the result of a \textbf{diffeological gluing of $X_1$ to $X_2$}.

There are two natural inclusions into the space $X_1\cup_f X_2$, whose ranges cover it. These are given by the maps $i_1:(X_1\setminus Y)\hookrightarrow(X_1\sqcup X_2)\to X_1\cup_f X_2$, where the 
second arrow stands for the natural projection onto the quotient space, and $i_2:X_2\hookrightarrow(X_1\sqcup X_2)\to X_1\cup_f X_2$. They are clearly bijective; furthermore, they are inductions (see 
\cite{pseudometric-pseudobundle}). The images $i_1(X_1\setminus Y)$ and $i_2(X_2)$ are disjoint and yield a covering of $X_1\cup_f X_2$, which is useful for constructing maps on/into $X_1\cup_f X_2$.

The plots of $X_1\cup_f X_2$ admit the following local description. Let $p:U\to X_1\cup_f X_2$ be a plot; then for every $u\in U$ there is
a neighborhood $U'\subset U$ of $u$ such that the restriction of $p$ on $U'$ lifts to a plot $p':U'\to(X_1\sqcup X_2)$. Since locally every plot of $X_1\sqcup X_2$ is a plot of either $X_1$ or $X_2$, up
to restricting it to a connected component $U''$ of $U'$ , there exists either a plot $p_1:U''\to X_1$ or a plot $p_2:U''\to X_2$ such that $p|_{U''}$ lifts to, respectively, $p_1$ or $p_2$. Furthermore, if
$p|_{U''}$ lifts to $p_2$ then its actual form is $p|_{U''}=i_2\circ p_2$, whereas if it lifts to $p_1$, its actual form is then as follows:
$$p|_{U''}(u'')=\left\{\begin{array}{ll} i_1(p_1(u'')) & \mbox{if }p_1(u'')\in X_1\setminus Y,\\ i_2(f(p_1(u''))) & \mbox{if }p_1(u'')\in Y. \end{array}\right.$$

The operation of gluing of two pseudo-bundles consists in performing twice the gluing of diffeological spaces, once for the total spaces, and the second time for the base spaces; the two gluing maps must be 
consistent with each other for the result to be a pseudo-bundle. Specifically, let $\pi_1:V_1\to X_1$ and $\pi_2:V_2\to X_2$ be two diffeological vector pseudo-bundles, let $f:X_1\supset Y\to X_2$ be a smooth 
map, and let $\tilde{f}:\pi_1^{-1}(Y)\to\pi_2^{-1}(f(Y))$ be any smooth lift of $f$ such that the restriction of $\tilde{f}$ on each fibre $\pi_1^{-1}(y)$ for $y\in Y$ is linear. Consider the diffeological spaces 
$V_1\cup_{\tilde{f}}V_2$ and $X_1\cup_f X_2$; since $\tilde{f}$ is a lift of $f$, the pseudo-bundle projections $\pi_1$ and $\pi_2$ yield a well-defined map, denoted by $\pi_1\cup_{(\tilde{f},f)}\pi_2$, from 
$V_1\cup_{\tilde{f}}V_2$ to $X_1\cup_f X_2$. Furthermore, by the linearity assumption on $\tilde{f}$ (see \cite{pseudobundles}), the map 
$$\pi_1\cup_{(\tilde{f},f)}\pi_2:V_1\cup_{\tilde{f}}V_2\to X_1\cup_f X_2$$ is itself a diffeological vector pseudo-bundle. The gluing is usually well-behaved with respect to the usual operations on vector 
pseudo-bundles (see \cite{pseudobundles}, \cite{pseudometric-pseudobundle}), with the one exception being the operation of taking the dual pseudo-bundle.

\begin{defn}
Let $\pi_1:V_1\to X_1$ and $\pi_2:V_2\to X_2$ be two diffeological vector pseudo-bundles, and let $(\tilde{f},f)$ be a gluing of the former to the latter such that $f$ is a diffeomorphism. Suppose that each of 
$V_1$, $V_2$ admits a pseudo-metric; let $g_i$ be a chosen pseudo-metric on $V_i$ for $i=1,2$. We say that $g_1$ and $g_2$ are \textbf{compatible} (with the gluing along $(\tilde{f},f)$) if for every $y\in Y$ 
and for all $v_1,v_2\in\pi_1^{-1}(y)$ we have
$$g_1(y)(v_1,v_2)=g_2(f(y))(\tilde{f}(v_1),\tilde{f}(v_2)).$$ Let $g_1$ and $g_2$ be compatible; the \textbf{induced pseudo-metric} on $V_1\cup_{\tilde{f}}V_2$ is the map
$\tilde{g}:X_1\cup_f X_2\to(V_1\cup_{\tilde{f}}V_2)^*\otimes(V_1\cup_{\tilde{f}}V_2)^*$ defined by 
$$\tilde{g}(x)(\cdot,\cdot)=\left\{\begin{array}{ll}
g_1(i_1^{-1}(x))(j_1^{-1}(\cdot),j_1^{-1}(\cdot)) & \mbox{for }x\in i_1(X_1\setminus Y) \\
g_2(i_2^{-1}(x))(j_2^{-1}(\cdot),j_2^{-1}(\cdot)) & \mbox{for }x\in i_2(X_2). \end{array}\right.$$ 
\end{defn}

The total and the base space of the new pseudo-bundle are both the results of a diffeological gluing, so everything we have said about it applies to each of them. In particular, there are the two pairs 
of standard inductions, which are denoted, as before, by $i_1,i_2$ for the base space $X_1\cup_f X_2$ and by $j_1,j_2$ for the total space $V_1\cup_{\tilde{f}}V_2$, that is, 
$j_1: (V_1\setminus\pi_1^{-1}(Y))\hookrightarrow V_1\sqcup V_2\to V_1\cup_{\tilde{f}}V_2$  and $j_2:V_2\hookrightarrow V_1\sqcup V_2\to V_1\cup_{\tilde{f}}V_2$.

\begin{rem}
Although we will not be able to really consider this in the present paper, we briefly mention how the gluing procedure provides a natural context for notions such as
$\delta$-functions. What we mean by this is the following. Let $X_1\subset\matR^2$ be the $x$-axis, let $X_2=\{(0,1)\}$, and let the gluing map $f:Y=\{(0,0)\}\to\{(0,1)\}$ be the obvious map. Let
$p:\matR\to X_1\cup_f X_2$ be the map defined by $p(x)=i_1(x,0)$ for $x\neq 0$ and $p(0)=i_2(0,1)$. By definition of the gluing diffeology, this is a plot of $X_1\cup_f X_2$ (and so an instance of a
smooth function into it). We now can observe that $p$ can be seen as the $\delta$-function $\matR\to\matR$, by projecting both $X_1$ and $X_2$ onto the $y$-axis of their ambient $\matR^2$. More
precisely, let $\mbox{pr}_y:\matR^2\to\matR$ be the projection of $\matR^2$ onto the $y$-axis, that is, $\mbox{pr}_y(x,y)=y$, and let $h:X_1\cup_f X_2\to\matR$ be given by
$$h(\tilde{x})=\left\{\begin{array}{ll} \mbox{pr}_y(i_1^{-1}(\tilde{x})) & \mbox{if }\tilde{x}\in i_1(X_1\setminus Y),\\ \mbox{pr}_y(i_2^{-1}(\tilde{x})) & \mbox{if }\tilde{x}\in i_2(X_2). \end{array}\right.$$ 
As follows from the definitions of $i_1$ and $i_2$, this defines $h$ on the entire $X_1\cup_f X_2$. Observe finally that the composition $h\circ p$ is indeed the usual $\delta$-function, \emph{i.e.}, 
the function $\delta$ given by $\delta(x)=0$ if $x\neq 0$ and $\delta(0)=1$.
\end{rem}

\subsection{Diffeological $1$-forms and gluing}

For diffeological spaces, there exists a rather well-developed theory of differential $k$-forms on them (see \cite{iglesiasBook}, Chapter 6, for a detailed exposition); we recall the main notions for the
case $k=1$.

\subsubsection{The pseudo-bundle of differential $1$-forms $\Lambda^1(X)$}

A \textbf{differential $1$-form on a diffeological space $X$} is defined by assigning to each plot  $p:\matR^k\supset U\to X$ a (usual) differential $1$-form such that this assignment satisfies the following smooth
compatibility condition. If $q:U'\to X$ is another plot of $X$ such that there exists a usual smooth map $F:U'\to U$ with $q=p\circ F$ then $\omega(q)(u')=F^*\left(\omega(p)(u)\right)$. Let now $f:X\to\matR$ be 
a diffeologically smooth function on it; recall that  this means that for every plot $p:U\to X$ the composition $f\circ p:U\to\matR$ is smooth in the usual sense, therefore $d(f\circ p)$ is a differential form on $U$. 
It is quite easy to see that the assignment $p\mapsto d(f\circ p)=:\omega(p)$ is a differential $1$-form on $X$; it is called, as usual, the \textbf{differential of $f$}.

The set of all differential $1$-forms on $X$ is denoted by $\Omega^1(X)$; it has an obvious vector space structure, with pointwise addition and scalar multiplication, and carries a natural functional diffeology 
with respect to which it becomes a diffeological vector space and that is characterized by the following condition: a map $q:U'\to\Omega^1(X)$ is a plot of $\Omega^1(X)$ if and only if for every plot $p:U\to X$ 
the map $U'\times U\to\Lambda^1(\matR^n)$ given by $(u',u)\mapsto(q(u')(p))(u)$ is smooth, where $U\subset\matR^n$. 

Let $X$ be a diffeological space, and let $x$ be a point of it. Denote by $\Omega_x^1(X)$ the subspace of all forms \textbf{vanishing at $x$}, namely, precisely the forms $\omega$ satisfying the following 
condition: for every plot $p:U\to X$ such that $U\ni 0$ and $p(0)=x$, the form $\omega(p)$ is the zero section of $\Lambda^1(U)$. The union
$$\bigcup_{x\in X}\{x\}\times\Omega_x^1(X)$$ is a (diffeological) sub-bundle of the trivial pseudo-bundle $X\times\Omega^1(X)$.  The \textbf{pseudo-bundle $\Lambda^1(X)$} is the corresponding quotient 
pseudo-bundle:
$$\Lambda^1(X):=\left(X\times\Omega^1(X)\right)/\left(\bigcup_{x\in X}\{x\}\times\Omega_x^1(X)\right).$$ 

The corresponding quotient projection is denoted by  
$$\pi^{\Omega,\Lambda}:X\times\Omega^1(X)\to\Lambda^1(X).$$ The pseudo-bundle projection of $\Lambda^1(X)$ is denoted by
$$\pi^{\Lambda}:\Lambda^1(X)\to X.$$ is locally represented by a pair, consisting of a plot of $X$ and a plot of $\Omega^1(X)$ (with the same domain of definition). The fibre at $x\in X$ of the pseudo-bundle 
$\Lambda^1(X)$ is denoted by $\Lambda_x^1(X)$, and we have
$$\Lambda_x^1(X)\cong\Omega^1(X)/\Omega_x^1(X).$$

\subsubsection{The behavior under gluing}

Let $X_1$ and $X_2$ be two diffeological spaces, and let $f:X_1\supseteq Y\to X_2$ be a smooth map. Let 
$$\pi:X_1\sqcup X_2\to X_1\cup_f X_2$$ be the quotient projection. We recall first the image of the pullback map (see 6.38 in \cite{iglesiasBook}) 
$$\pi^*:\Omega^1(X_1\cup_f X_2)\to\Omega^1(X_1\sqcup X_2)\cong\Omega^1(X_1)\times\Omega^1(X_2).$$ Note that $\pi^*$ is injective but typically not surjective.

\paragraph{The characteristic maps $\tilde{\rho}_1^{\Lambda}$ and $\tilde{\rho}_2^{\Lambda}$} In addition to the induction $i_2:X_2\hookrightarrow X_1\cup_f X_2$, consider the map 
$\tilde{i}_1:X_1\to X_1\cup_f X_2$ given as a composition of the inclusion $X_1\hookrightarrow X_1\sqcup X_2$ with the quotient projection $X_1\sqcup X_2\to X_1\cup_f X_2$. The corresponding 
pullback maps $\tilde{i}_1^*$ and $i_2^*$ induce well-defined, smooth, and linear maps
\begin{flushleft}
$\tilde{\rho}_1^{\Lambda}:\Lambda^1(X_1\cup_f X_2)\supset(\pi^{\Lambda})^{-1}(\tilde{i}_1(X_1))\to\Lambda^1(X_1)$ and
\end{flushleft}
\begin{flushright}
$\tilde{\rho}_2^{\Lambda}:\Lambda^1(X_1\cup_f X_2)\supset(\pi^{\Lambda})^{-1}(i_2(X_2))\to\Lambda^1(X_2)$.
\end{flushright}
See \cite{forms-gluing} for details.

\paragraph{The extendibility conditions $i^*(\Omega^1(X_1))=(f^*j^*)(\Omega^1(X_2))$ and $\calD_1^{\Omega}=\calD_2^{\Omega}$} Let $i:Y\hookrightarrow X_1$ and $j:f(Y)\hookrightarrow X_2$ be the 
natural inclusions. Denote by $\calD_1^{\Omega}$ the diffeology on $\Omega^1(Y)$ obtained as the pushforward of the standard diffeology of $\Omega^1(X_1)$ by the pullback map $i^*$. Likewise, let 
$\calD_2^{\Omega}$ be the pushforward of the standard diffeology on $\Omega^1(X_2)$ by the map $f^*j^*$. The conditions $i^*(\Omega^1(X_1))=(f^*j^*)(\Omega^1(X_2))$ and 
$\calD_1^{\Omega}=\calD_2^{\Omega}$ will be needed throughout the paper (a lot of statements depend on them). Notice that the latter condition is stronger than the former.

\paragraph{Compatibility of elements of $\Lambda^1(X_1)$ and $\Lambda^1(X_2)$} We will also need a certain compatibility notion for elements of $\Lambda^1(X_1)$ and
$\Lambda^1(X_2)$. This compatibility is relative to the map $f$ and applies to elements of fibres over the domain of gluing.

\begin{defn}\label{comp:elements:of:lambda:defn}
Two forms $\omega_1\in\Omega^1(X_1)$ and $\omega_2\in\Omega^1(X_2)$ are \textbf{compatible} if for any plot $p_1$ of $X_1$ whose range is contained in $Y$ we have that 
$$\omega_1(p_1)=\omega_2(f\circ p_1).$$ Let now $y\in Y$. Two cosets $\omega_1+\Omega_y^1(X_1)$ and $\omega_2+\Omega_{f(y)}^1(X_2)$ are said to be \textbf{compatible} if for any 
$\omega_1'\in\Omega_y^1(X_1)$ and for any $\omega_2'\in\Omega_{f(y)}^1(X_2)$ the forms $\omega_1+\omega_1'$ and $\omega_2+\omega_2'$ are compatible.
\end{defn}

It is useful to note (\cite{forms-gluing}) that $\omega_1\in\Omega^1(X_1)$ and $\omega_2\in\Omega^1(X_2)$ are compatible if and only if 
$$i^*(\omega_1)=(f^*j^*)(\omega_2).$$

\paragraph{The individual fibres of $\Lambda^1(X_1\cup_f X_2)$} Assuming that $f$ is a diffeomorphism, the fibres of $\Lambda^1(X_1\cup_f X_2)$ can be fully described. It turns out that any of them
is diffeomorphic to either a fibre of one of the factors or to the direct sum of such. This distinction depends on whether the fibre is at a point of the domain of gluing or outside of it.

\begin{thm}\label{fibres:lambda:f:diffeo:thm} \emph{(\cite{forms-gluing})}
Let $X_1$ and $X_2$ be two diffeological spaces, let $f:X_1\supseteq Y\to X_2$ be a diffeomorphism of its domain with its image such that $\calD_1^{\Omega}=\calD_2^{\Omega}$, and let 
$x\in X_1\cup_f X_2$. Then:
$$\Lambda_x^1(X_1\cup_f X_2)\cong\left\{\begin{array}{ll}
\Lambda_{i_1^{-1}(x)}^1(X_1) & \mbox{if }x\in i_1(X_1\setminus Y)\\ 
\Lambda_{\tilde{i}^{-1}(x)}^1(X_1)\oplus_{comp}\Lambda_{i_2^{-1}(x)}^1(X_2) & \mbox{if }x\in i_2(f(Y)) \\
\Lambda_{i_2^{-1}(x)}^1(X_2) & \mbox{if }x\in i_2(X_2\setminus f(Y)).
\end{array}\right.$$
\end{thm}

\begin{example}
Let $X_1$ and $X_2$ be two diffeological spaces, and let $x_i\in X_i$ be a point, for $i=1,2$; let $f:\{x_1\}\to\{x_2\}$ be the obvious map. Then $X_1\cup_f X_2$ is the usual wedge $X_1\vee X_2$ of
$X_1$ and $X_2$. Since any diffeological form assigns the zero value to any constant plot, any two forms on, respectively, $X_1$ and $X_2$, are automatically compatible. Therefore
$$\Omega^1(X_1\cup_f X_2)=\Omega^1(X_1\vee X_2)\cong\Omega^1(X_1)\times\Omega^1(X_2).$$ On the other hand, the fibre of $\Lambda^1(X_1\vee X_2)$ at any point
$x_i'\in X_i\subset X_1\vee X_2$ is $\Lambda_{x_i'}^1(X_i)$, except for the wedge point $x=[x_1]=[x_2]$, where it is the direct product $\Lambda_{x_1}^1(X_1)\times\Lambda_{x_2}^1(X_2)$.
\end{example}

\paragraph{On the diffeology of $\Lambda^1(X_1\cup_f X_2)$} There is first of all the following characterization of the diffeology of $\Lambda^1(X_1\cup_f X_2)$.

\begin{thm} \emph{(\cite{forms-gluing})}
If $\calD_1^{\Omega}=\calD_2^{\Omega}$ then the diffeology of $\Lambda^1(X_1\cup_f X_2)$ is the coarsest one such that both $\tilde{\rho}_i^{\Lambda}$ are smooth. Furthermore, the maps 
$\tilde{\rho}_1^{\Lambda}$ and $\tilde{\rho}_2^{\Lambda}$ are subductions.
\end{thm}

More technical details (that later on we will make use of) appear in the following statement.

\begin{thm}\label{decomposition:lambda:f:diffeo:thm}
Let $X_1$ and $X_2$ be two diffeological spaces, and let $f:X_1\supseteq Y\to X_2$ be a diffeomorphism of its domain with its image such that $\calD_1^{\Omega}=\calD_2^{\Omega}$. Let 
$\pi^{\Lambda}:\Lambda^1(X_1\cup_f X_2)\to X_1\cup_f X_2$, $\pi_1^{\Lambda}:\Lambda^1(X_1)\to X_1$, and $\pi_2^{\Lambda}:\Lambda^1(X_2)\to X_2$ be the pseudo-bundle projections. Then;
\begin{itemize}
\item the restriction of $\tilde{\rho}_1^{\Lambda}$ to $(\pi^{\Lambda})^{-1}(i_1(X_1\setminus Y))$ is a diffeomorphism
$$\Lambda^1(X_1\cup_f X_2)\supseteq(\pi^{\Lambda})^{-1}(i_1(X_1\setminus Y))\to(\pi_1^{\Lambda})^{-1}(X_1\setminus Y)\subset\Lambda^1(X_1)$$
for the appropriate subset diffeologies;

\item the restriction of $\tilde{\rho}_2^{\Lambda}$ to $(\pi^{\Lambda})^{-1}(i_2(X_2\setminus f(Y)))$ is a diffeomorphism 
$$\Lambda^1(X_1\cup_f X_2)\supseteq(\pi^{\Lambda})^{-1}(i_2(X_2\setminus f(Y)))\to(\pi_2^{\Lambda})^{-1}(X_2\setminus f(Y));$$ 

\item the direct sum of the restrictions of $\tilde{\rho}_1^{\Lambda}$ and $\tilde{\rho}_2^{\Lambda}$ to $i_2(f(Y))$, written as $\tilde{\rho}_1^{\Lambda}\oplus\tilde{\rho}_2^{\Lambda}$, is a diffeomorphism
$$\Lambda^1(X_1\cup_f X_2)\supseteq(\pi^{\Lambda})^{-1}(i_2(f(Y)))\to(\pi_1^{\Lambda})^{-1}(Y)\oplus_{comp}(\pi_2^{\Lambda})^{-1}(f(Y)).$$
\end{itemize}
\end{thm}

\section{Sections of diffeological pseudo-bundles}

In this section we consider the space $C^{\infty}(X,V)$ of smooth sections of a given finite-dimensional diffeological vector pseudo-bundle. For non-standard diffeologies on one or both of $X$ and $V$, 
this space may easily be of infinite dimension; immediately below we provide an example of this. On the other hand, when it is the spaces themselves that are non-standard (say, they are not 
topological manifolds), the space of sections might be finite-dimensional, as we illustrate via the study of the behavior of the space of sections under gluing of pseudo-bundles, 
where most of the effort has to be spent on the case when the gluing is performed along a non-diffeomorphism $f$. In this regard, we obtain the answer in the most general case, showing that the space 
of sections of the result of gluing is always a smooth surjective imageof a subspace of the direct product of the spaces of sections of the factors (in particular, the finiteness of the dimension is preserved, 
\emph{i.e.}, the existence of local bases, meaning that if the spaces of sections of the factors are finite-dimensional then so is the space of sections of the result of gluing).

Observe that we discuss in fact only the case of global sections, because this is not really different from restricting ourselves to the local case. Indeed, in the natural topology underlying a diffeological 
structure, the so-called D-topology (introduced in \cite{iglFibre}), the open sets can be of any form. What this implies at the moment is that there is no fixed local shape for diffeological pseudo-bundle, or, 
said differently, any diffeological vector pseudo-bundle can appear as the restriction of a larger pseudo-bundle to a D-open (open in D-topology) neighborhood of a fibre.

The final conclusion of this section is that there is a natural smooth surjective map (a subduction, in fact)
$$\mathcal{S}:C_{(f,\tilde{f})}^{\infty}(X_1,V_1)\times_{comp}C^{\infty}(X_2,V_2)\to C^{\infty}(X_1\cup_f X_2,V_1\cup_{\tilde{f}}V_2),$$ where 
$C_{(f,\tilde{f})}^{\infty}(X_1,V_1)\leqslant C^{\infty}(X_1,V_1)$ is the subspace of the so-called \emph{$(f,\tilde{f})$-invariant} sections (these are sections $s$ such that having $f(y)=f(y')$ implies that 
$\tilde{f}(s(y))=\tilde{f}(s(y'))$), and $C_{(f,\tilde{f})}^{\infty}(X_1,V_1)\times_{comp}C^{\infty}(X_2,V_2)$ is the subset of the direct product $C_{(f,\tilde{f})}^{\infty}(X_1,V_1)\times C^{\infty}(X_2,V_2)$ 
that consists of all \emph{compatible} pairs (a pair $(s_1,s_2)$ is compatible if $\tilde{f}\circ s_1=s_2\circ f$ wherever defined). The map $\mathcal{S}$ is an instance of a more general procedure of 
gluing compatible maps concurrently with gluing of their domains and their ranges (see \cite{pseudometric-pseudobundle} for the general case; the map $\mathcal{S}$ is described in detail below). Notice 
also that
$$C_{(f,\tilde{f})}^{\infty}(X_1,V_1)\times_{comp}C^{\infty}(X_2,V_2)=C^{\infty}(X_1,V_1)\times_{comp} C^{\infty}(X_2,V_2),$$ that is, if $(s_1,s_2)$ is a compatible pair then $s_1$ is necessarily 
$(f,\tilde{f})$-invariant. Also, $\mathcal{S}$ is a diffeomorphism if $f$ and $\tilde{f}$ are so.

\subsection{The space $C^{\infty}(X,V)$ over $C^{\infty}(X,\matR)$}

In the case of diffeological pseudo-bundles, the space of all smooth sections $C^{\infty}(X,V)$ may have infinite dimension over $C^{\infty}(X,\matR)$ when normally we would not expect it. To begin our 
consideration of the subject, we provide a simple example of this.

\begin{example}
Let $\pi:V\to X$ be the projection of $V=\matR^3$ onto its first coordinate; thus, $X$ is $\matR$, which we endow with the standard diffeology. Endow $V$ with the pseudo-bundle diffeology
generated by the plot $\matR^2\ni(u,v)\mapsto(u,0,|v|)$; this diffeology is a product diffeology for the decomposition $\matR^3=\matR\times\matR^2$ into the direct product of the standard
$\matR$ with $\matR^2$ carrying the vector space diffeology generated by the plot $v\mapsto(0,|v|)$. In this case the space $C^{\infty}(X,V)$ of smooth sections of the pseudo-bundle $\pi$
has infinite dimension over $C^{\infty}(X,\matR)=C^{\infty}(\matR,\matR)$; let us explain why.
\end{example}

\begin{proof}
Since the diffeology of $X$ is standard, the ring $C^{\infty}(X,\matR)$ includes the usual smooth functions $\matR\to\matR$ only. The diffeology of $V$ is actually a vector space diffeology
generated by the plot $(u,v)\mapsto(u,0,|v|)$; an arbitrary plot of it has therefore the form
$$\matR^{l+m+n}\supseteq U\ni(x,y,z)\mapsto(f_1(x),f_2(y),g_0(z)+g_1(z)|h_1(z)|+\ldots+g_k(z)|h_k(z)|),$$
where $U$ is a domain, and $f_1:\matR^l\subseteq U_x\to\matR$, $f_2:\matR^m\supseteq U_y\to\matR$ and $g_0,g_1,\ldots,g_k,h_1,\ldots,h_k:\matR^n\supseteq U_z\to\matR$ are some ordinary 
smooth functions. Hence any smooth section $s\in C^{\infty}(X,V)$ has (at least locally) form
$$s(x)=(x,f(x),g_0(x)+g_1(x)|h_1(x)|+\ldots+g_k(x)|h_k(x)|)$$ for some ordinary smooth functions $f,g_0,g_1,\ldots,g_k,h_1,\ldots,h_k:\matR\supseteq U\to\matR$; and \emph{vice versa} every such 
expression corresponds locally to a smooth section $X\to V$ (and can be extended, by a standard partition-of-unity argument, to a section in $C^{\infty}(X,V)$). Since $g_i$ and $h_i$ are any smooth 
functions at all, and they can be in any finite number, for any finite arbitrarily long collection $x_1,\ldots,x_k\in\matR$ there is a diffeologically smooth section $s$ that, seen as a usual map 
$\matR\to\matR^3$, is non-differentiable precisely at the points $x_1,\ldots,x_k$ (and smooth outside of them). Thus, it is impossible that all such sections be linear combinations over 
$C^{\infty}(\matR,\matR)$ of the same finite set of (at least continuous) functions $\matR\to\matR^3$.
\end{proof}

Our main interest thus is when the space of sections turns out to be finite-dimensional. We thus concentrate, in the sections that follow, on the behavior of this space under the operation of gluing.

\subsection{The space $C^{\infty}(X_1\cup_f X_2,V_1\cup_{\tilde{f}}V_2)$ relative to $C_{(f,\tilde{f})}^{\infty}(X_1,V_1)$ and $C^{\infty}(X_2,V_2)$}

Let $\pi_1:V_1\to X_1$ and $\pi_2:V_2\to X_2$ be two finite-dimensional diffeological vector pseudo-bundles, let $(\tilde{f},f)$ be a pair of smooth maps that defines a gluing between them, and let 
$Y\subset X_1$ be the domain of definition of $f$. Consider the three corresponding spaces of smooth sections, \emph{i.e.}, the spaces $C^{\infty}(X_1,V_1)$, $C^{\infty}(X_2,V_2)$, and 
$C^{\infty}(X_1\cup_f X_2,V_1\cup_{\tilde{f}}V_2)$. The latter space can be reconstructed from the former two by using the notion of gluing of compatible smooth maps, as it  appears in
 \cite{pseudometric-pseudobundle}.

\subsubsection{Compatible sections}

Consider a pair $\varphi_1:X_1\to Z_1$ and $\varphi_2:X_2\to Z_2$ of smooth maps between some diffeological spaces that are, in turn, endowed with fixed smooth maps $f:X_1\supset Y\to X_2$ 
and $g:\varphi_1(Y)\to Z_2$. We say that $\varphi_1$ and $\varphi_2$ are \textbf{$(f,g)$-compatible} if $g\circ\varphi_1=\varphi_2\circ f$ wherever defined. This allows to define an obvious map
$$\varphi_1\cup_{(f,g)}\varphi_2:X_1\cup_f X_2\to Z_1\cup_g Z_2,$$ which is smooth for the gluing diffeologies on $X_1\cup_f X_2$ and $Z_1\cup_g Z_2$ (see \cite{pseudometric-pseudobundle}, 
Proposition 4.4). A pair of sections $s_1,s_2$ of two diffeological pseudo-bundles is then a particular instance of maps $\varphi_i$, with $Z_i$ being $V_i$, with the role of $g$ being played by $\tilde{f}$. 
Two such sections are \textbf{compatible} if $\tilde{f}\circ s_1=s_2\circ f$ on the whole domain of definition.

\subsubsection{Compatibility and $(f,\tilde{f})$-invariant sections}

Let $s_1\in C^{\infty}(X_1,V_1)$ be such that there exists a section $s_2\in C^{\infty}(X_2,V_2)$ compatible with it, that is, for all $y\in Y$ we have $\tilde{f}(s_1(y))=s_2(f(y))$; let $y'\in Y$ be a point 
such that $f(y)=f(y')$. The compatibility condition implies then that
$$\tilde{f}(s_1(y))=s_2(f(y))=s_2(f(y'))=\tilde{f}(s_1(y'));$$ thus, although $s_1(y)$ and $s_1(y')$ do not have to coincide, their images under $\tilde{f}$ necessarily do. This justifies the following definition, 
and an easy lemma that follows it.

\begin{defn}\label{f-tilde_f:invariant:sections:defn}
Let $\pi_1:V_1\to X_1$ be a diffeological vector pseudo-bundle, let $W$ and $Z$ be any two diffeological spaces, let $f:Y\to Z$ be a smooth map defined on a subset $Y$ of $X_1$,
and let $\tilde{f}:\pi_1^{-1}(Y)\to W$ be a lift of $f$ to $V_1$. A section $s_1\in C^{\infty}(X_1,V_1)$ of this pseudo-bundle is called \textbf{$(f,\tilde{f})$-invariant} if for any $y,y'\in Y$ such
that $f(y)=f(y')$ (in $Z$) we have that $\tilde{f}(s_1(y))=\tilde{f}(s_1(y'))$ (in $W$). A function $h\in C^{\infty}(X_1,\matR)$ is called \textbf{$f$-invariant} if $h(y)=h(y')$ for all $y,y'\in Y$
such that $f(y)=f(y')$.
\end{defn}

The lemma below follows immediately from what has been said prior to the definition.

\begin{lemma}\label{s_1:in:compatible:is:ff-invt:lem}
Let $\pi_1:V_1\to X_1$ and $\pi_2:V_2\to X_2$ be two diffeological vector pseudo-bundles, let $(\tilde{f},f)$ be a gluing between them, and let $s_1\in C^{\infty}(X_1,V_1)$
be such that there exists $s_2\in C^{\infty}(X_2,V_2)$ compatible with $s_1$. Then $s_1$ is $(f,\tilde{f})$-invariant.
\end{lemma}

Thus, we only need to take into consideration $(f,\tilde{f})$-invariant sections $X_1\to V_1$. The set of all of them is denoted by
$$C_{(f,\tilde{f})}^{\infty}(X_1,V_1)=\{s\in C^{\infty}(X_1,V_1)\,|\,s\mbox{ is }(f,\tilde{f})-\mbox{invariant}\}.$$ Let us now consider some properties of this set.

\begin{prop}\label{prod:by:f-invt:is:ff:invt:prop}
The set $C_{(f,\tilde{f})}^{\infty}(X_1,V_1)$ is closed with respect to the summation and multiplication by $f$-invariant functions.
\end{prop}

\begin{proof}
Given two $(f,\tilde{f})$-invariant sections $s_1,s_2\in C^{\infty}(X_1,V_1)$ and an $f$-invariant function $h\in C^{\infty}(X_1,\matR)$, we need to show that $s_1+s_2$ and $hs_1$ are again
$(f,\tilde{f})$-invariant. Let $y,y'\in Y$ be such that $f(y)=f(y')$. Then it follows from the linearity of $\tilde{f}$ on each fibre in its domain of definition that
$\tilde{f}((hs)(y))=\tilde{f}(h(y)s(y))=h(y)\tilde{f}(s(y))=h(y')\tilde{f}(s(y'))=\tilde{f}((hs)(y'))$, as wanted. Furthermore, by assumption $\tilde{f}(s_1(y))=\tilde{f}(s_1(y'))$ and $\tilde{f}(s_2(y))=\tilde{f}(s_2(y'))$. 
Since the restriction of $\tilde{f}$ on any fibre is linear, we have
$$\tilde{f}((s_1+s_2)(y))=\tilde{f}(s_1(y)+s_2(y))=\tilde{f}(s_1(y))+\tilde{f}(s_2(y))=\tilde{f}(s_1(y'))+\tilde{f}(s_2(y'))=\tilde{f}((s_1+s_2)(y')),$$ which completes the proof.
\end{proof}

\subsubsection{The map $\mathcal{S}:C_{(f,\tilde{f})}^{\infty}(X_1,V_1)\times_{comp}C^{\infty}(X_2,V_2)\to C^{\infty}(X_1\cup_f X_2,V_1\cup_{\tilde{f}}V_2)$}

The notion of compatibility of sections allows us to define the (partial) operation of gluing for smooth sections of the pseudo-bundles $\pi_1$ and $\pi_2$, through which we define the map announced in 
the title of the section.

\paragraph{The section $s_1\cup_{(f,\tilde{f})}s_2$ of $\pi_1\cup_{(\tilde{f},f)}\pi_2$} Let $s_1$ and $s_2$ be two compatible smooth sections of the pseudo-bundles $\pi_1$ and $\pi_2$ respectively. 
We define a section $s_1\cup_{(f,\tilde{f})}s_2:X_1\cup_f X_2\to V_1\cup_{\tilde{f}}V_2$ of the pseudo-bundle $\pi_1\cup_{(\tilde{f},f)}\pi_2$ by imposing
$$(s_1\cup_{(f,\tilde{f})}s_2)(x)=\left\{\begin{array}{ll}
s_1(i_1^{-1}(x)) & \mbox{for }x\in i_1(X_1\setminus Y),\mbox{ and}\\
s_2(i_2^{-1}(x)) & \mbox{for }x\in i_2(X_2).
\end{array}\right.$$
This turns out to be a smooth section of $\pi_1\cup_{(\tilde{f},f)}\pi_2$, as follows from Proposition 4.4 in \cite{pseudometric-pseudobundle}.

\paragraph{The induced map $\mathcal{S}$ into $C^{\infty}(X_1\cup_f X_2,V_1\cup_{\tilde{f}}V_2)$} Consider the direct product $C^{\infty}(X_1,V_1)\times C^{\infty}(X_2,V_2)$; let
$$C^{\infty}(X_1,V_1)\times_{comp}C^{\infty}(X_2,V_2)=\{(s_1,s_2)\,|\,s_i\in C^{\infty}(X_i,V_i),\,\,s_1,s_2\mbox{ are }(f,\tilde{f})-\mbox{compatible}\}.$$ The latter set is endowed with the subset diffeology 
relative to its inclusion into $C^{\infty}(X_1,V_1)\times C^{\infty}(X_2,V_2)$ (which in turn has the product diffeology coming from the functional diffeologies on each $C^{\infty}(X_i,V_i)$). Notice that 
by Lemma \ref{s_1:in:compatible:is:ff-invt:lem} 
$$C^{\infty}(X_1,V_1)\times C^{\infty}(X_2,V_2)=C_{(f,\tilde{f})}^{\infty}(X_1,V_1)\times_{comp}C^{\infty}(X_2,V_2).$$ The map
$$\mathcal{S}:C_{(f,\tilde{f})}^{\infty}(X_1,V_1)\times_{comp}C^{\infty}(X_2,V_2)\to C^{\infty}(X_1\cup_f X_2,V_1\cup_{\tilde{f}}V_2)$$ is defined by
$$\mathcal{S}(s_1,s_2)=s_1\cup_{(f,\tilde{f})}s_2;$$ it has the following property.

\begin{thm}\label{mathcal:S:is:smooth:thm} \emph{(\cite{pseudometric-pseudobundle})}
The map $\mathcal{S}$ is smooth, for the subset diffeology on $C_{(f,\tilde{f})}^{\infty}(X_1,V_1)\times_{comp}C^{\infty}(X_2,V_2)$ and the functional diffeology on 
$C^{\infty}(X_1\cup_f X_2,V_1\cup_{\tilde{f}}V_2)$.
\end{thm}

\subsubsection{Pseudo-bundles operations and the map $\mathcal{S}$}

Let $\pi_1:V_1\to X_1$ and $\pi_2:V_2\to X_2$ be two finite-dimensional diffeological vector pseudo-bundles, with $(\tilde{f},f)$ being a gluing between them, and let $\pi_1':V_1'\to X_1$ and 
$\pi_2':V_2'\to X_2$ be two other pseudo-bundles, with the same base spaces, with $(\tilde{f}',f)$ also a gluing between them.
\begin{itemize}
\item (\cite{pseudometric-pseudobundle}, Proposition 4.7) Let $s_i\in C^{\infty}(X_i,V_i)$ for $i=1,2$ be $(f,\tilde{f})$-compatible sections, and let $h_i\in C^{\infty}(X_i,\matR)$ be such that $h_2(f(y))=h_1(y)$ 
for all $y\in Y$. Then $$\left(h_1\cup_f h_2\right)\left(s_1\cup_{(f,\tilde{f})}s_2\right)=(h_1s_1)\cup_{(f,\tilde{f})}(h_2s_2).$$
\item (\cite{pseudometric-pseudobundle}, Proposition 4.8) Let $s_i\in C^{\infty}(X_i,V_i)$ and $s_i'\in C^{\infty}(X_i,V_i')$ be such that $s_1,s_2$ are $(f,\tilde{f})$-compatible, while $s_1',s_2'$ are 
$(f,\tilde{f}')$-compatible. Then $$\left(s_1\cup_{(f,\tilde{f})}s_2\right)\otimes\left(s_1'\cup_{(f,\tilde{f}')}s_2'\right)=(s_1\otimes s_1')\cup_{(f,\tilde{f}\otimes\tilde{f}')}(s_2\otimes s_2').$$
\end{itemize}
In addition to these, we now prove that $\mathcal{S}$ is additive with respect to the direct sum structure on $C^{\infty}(X_1,V_1)\times C^{\infty}(X_2,V_2)$, of which 
$C^{\infty}(X_1,V_1)\times_{comp}C^{\infty}(X_2,V_2)$ is then a subspace.

\begin{lemma}
If $s_1,t_1\in C^{\infty}(X_1,V_1)$ and $s_2,t_2\in C^{\infty}(X_2,V_2)$ are such that both $(s_1,s_2)$ and $(t_1,t_2)$ are $(f,\tilde{f})$-compatible pairs, then also $(s_1+t_1,s_2+t_2)$ is a 
$(f,\tilde{f})$-compatible pair, and $$(s_1+t_1)\cup_{(f,\tilde{f})}(s_2+t_2)=s_1\cup_{(f,\tilde{f})}s_2+t_1\cup_{(f,\tilde{f})}t_2.$$
\end{lemma}

\begin{proof}
Let $y\in Y$; then
$$\tilde{f}(s_1(y)+t_1(y))=\tilde{f}(s_1(y))+\tilde{f}(t_1(y))=s_2(f(y))+t_2(f(y)),$$ so $s_1+t_1$ and $s_2+t_2$ are $(f,\tilde{f})$-compatible. Now, by definition
\begin{flushleft}
$\left((s_1+t_1)\cup_{(f,\tilde{f})}(s_2+t_2)\right)(x)=\left\{\begin{array}{l}(s_1+t_1)(i_1^{-1}(x)) \\ (s_2+t_2)(i_2^{-1}(x)) \end{array}\right.=
\left\{\begin{array}{l} s_1(i_1^{-1}(x))+t_1(i_1^{-1}(x)) \\ s_2(i_2^{-1}(x))+t_2(i_2^{-1}(x)) \end{array}\right.=$
\end{flushleft}
\begin{flushright}
$=\left\{\begin{array}{l}s_1(i_1^{-1}(x)) \\ s_2(i_2^{-1}(x)) \end{array}\right. +\left\{\begin{array}{l}t_1(i_1^{-1}(x)) \\ t_2(i_2^{-1}(x))\end{array}\right.=
(s_1\cup_{(f,\tilde{f})}s_2)(x)+(t_1\cup_{(f,\tilde{f})}t_2)(x)$,
\end{flushright} where in each two-part formula the first line applies to $x\in i_1(X_1\setminus Y)$ and the second line, to $x\in i_2(X_2)$. The final equality that we obtain is precisely the first item in the 
statement of the lemma, so we are done.
\end{proof}

\subsection{The map $\mathcal{S}$ is a subduction}

In this section we show that $\mathcal{S}$ is a subduction of the space $C_{(f,\tilde{f})}^{\infty}(X_1,V_1)\times_{comp} C^{\infty}(X_2,V_2)$ onto the space $C^{\infty}(X_1\cup_f X_2,V_1\cup_{\tilde{f}}V_2)$.

\subsubsection{Gluing along diffeomorphisms}

Clearly in this case $C^{\infty}(X_1,V_1)=C_{(f,\tilde{f})}^{\infty}(X_1,V_1)$. Furthermore, it is rather easy to show that $\mathcal{S}$ is actually a diffeomorphism; we need a preliminary statement first.

\begin{lemma}
The maps $\tilde{i}_1:X_1\to X_1\cup_f X_2$ and $\tilde{j}_1:V_1\to V_1\cup_{\tilde{f}}V_2$ defined as the compositions of the natural inclusions into $X_1\sqcup X_2$ and $V_1\sqcup V_2$ with the 
corresponding quotient projections, are diffeomorphisms with their images.
\end{lemma}

\begin{proof}
That $\tilde{i}_1$ and $\tilde{j}_1$ are bijections with their respective images is immediately obvious. Furthermore, they are always smooth, since they are compositions of two smooth maps. Finally, 
their inverses are smooth by the definition of the gluing diffeology as a pushforward one (the assumption that $f$ and $\tilde{f}$ are diffeomorphisms is only needed for the existence of these inverses).
\end{proof}

We are now ready to prove the following.

\begin{prop}\label{space:of:sections:case:of:two:diffeo:prop}
If both $\tilde{f}$ and $f$ are diffeomorphisms of their domains with their images, the map $\mathcal{S}$ is a diffeomorphism
$$C^{\infty}(X_1,V_1)\times_{comp}C^{\infty}(X_2,V_2)\to C^{\infty}(X_1\cup_f X_2,V_1\cup_{\tilde{f}}V_2).$$
\end{prop}

\begin{proof}
The inverse of $\mathcal{S}$ is obtained by assigning to each section $s\in C^{\infty}(X_1\cup_f X_2,V_1\cup_{\tilde{f}}V_2)$ the pair
$$(\tilde{j}_1^{-1}\circ s|_{\tilde{i}_1}(X_1)\circ\tilde{i}_1,j_2^{-1}\circ s|_{i_2(X_2)}\circ i_2),$$ where $\tilde{i}_1:X_1\to X_1\cup_f X_2$ and $\tilde{j}_1:V_1\to V_1\cup_{\tilde{f}}V_2$ are the just-mentioned 
inclusions of $X_1$ and $V_1$ into $X_1\cup_f X_2$ and $V_1\cup_{\tilde{f}}V_2$ respectively. It follows that $\tilde{j}_1^{-1}\circ s|_{\tilde{i}_1}(X_1)\circ\tilde{i}_1=:s_1\in C^{\infty}(X_1,V_1)$, 
whereas $j_2^{-1}\circ s|_{i_2(X_2)}\circ i_2=:s_2\in C^{\infty}(X_2,V_2)$ holds even without extra assumptions.

Let us formally check that $s_1$ and $s_2$ are compatible. Let $y\in Y$; then $\tilde{f}(s_1(y))=\tilde{f}(\tilde{j}_1^{-1}(s(\tilde{i}_1(y))))$, and $s_2(f(y))=j_2(s(i_2(f(y))))$. Since $\tilde{i}_1(y)=i_2(f(y))$ 
by construction, we have $\tilde{f}(s_1(y))=\tilde{f}(\tilde{j}_1^{-1}(s(i_2(f(y)))))$, and it suffices to note that $\tilde{f}\circ\tilde{j}_1^{-1}=j_2$ on the entire $\pi_2^{-1}(i_2(f(Y)))$.
\end{proof}

\subsubsection{The pseudo-bundle $\pi_1^{(\tilde{f},f)}:V_1^{\tilde{f}}\to X_1^f$ of $(\tilde{f},f)$-equivalence classes}

In the case when $f$ and $\tilde{f}$ are not diffeomorphisms, we need an auxiliary construction, that of the pseudo-bundle mentioned in the title of the section. Its base space and its total space are obtained 
from $X_1$ and $V_1$ respectively by natural quotientings, given by $f$ and $\tilde{f}$, and the pseudo-bundle projection is induced by $\pi_1$.

\paragraph{The spaces $X_1^f$ and $V_1^{\tilde{f}}$} The base space $X_1^f$ is defined as the diffeological quotient $X_1/_{\sim^f}$, where the equivalence relation $\sim^f$ is given by
$$y_1\sim^f y_2\Leftrightarrow f(y_1)=f(y_2).$$ Likewise, the space $V_1^{\tilde{f}}$ is the quotient of $V_1$ by the equivalence relation $\sim^{\tilde{f}}$ that is analogous to $\sim^f$ and is given by
$$v_1\sim^{\tilde{f}}v_2\Leftrightarrow\tilde{f}(v_1)=\tilde{f}(v_2).$$ The two quotient projections are denoted respectively by $\chi_1^f:X_1\to X_1^f$ and by $\chi_1^{\tilde{f}}:V_1\to V_1^{\tilde{f}}$. 
The space $X_1^f$ is endowed with the map $f_{\sim}:\chi_1^f(Y)\to X_2$, and the space $V_1^{\tilde{f}}$, with the map $\tilde{f}_{\sim}:\chi_1^{\tilde{f}}(\pi_1^{-1}(Y))\to V_2$. These are the pushforwards 
of, respectively, $f$ and $\tilde{f}$ by the quotient projections $\chi_1^f$ and $\chi_1^{\tilde{f}}$. If either of $f$, $\tilde{f}$ is a subduction then the corresponding induced map $f_{\sim}$ or $\tilde{f}_{\sim}$ 
is a diffeomorphism with its image.

\paragraph{The map $\pi_1^{(\tilde{f},f)}$} We now show that the induced pseudo-bundle projection $V_1^{\tilde{f}}\to X_1^f$ is indeed a pseudo-bundle.

\begin{lemma}\label{gluing:p-bundles:yields:p-bundle:lem}
There is a well-defined smooth map $\pi_1^{(\tilde{f},f)}:V_1^{\tilde{f}}\to X_1^f$ such that 
$$\chi_1^f\circ\pi_1=\pi_1^{(\tilde{f},f)}\circ\chi_1^{\tilde{f}}.$$ Furthermore, for any $x\in X_1^f$ the pre-image $(\pi_1^{(\tilde{f},f)})^{-1}(x)\subset V_1^{\tilde{f}}$ inherits from $V_1$ the structure of a 
diffeological vector space, with respect to which the corresponding restriction of $\tilde{f}_{\sim}$, when it is defined, is a linear map.
\end{lemma}

\begin{proof}
That $\pi_1^{(\tilde{f},f)}$ is uniquely defined by the condition given, follows from $\chi_1^{\tilde{f}}$ being surjective, and that it is well-defined follows from $\sim^{\tilde{f}}$ being a fibrewise equivalence 
relation. That the pre-image, in $V_1^{\tilde{f}}$, of any point $x\in X_1^f$ under the map $\pi_1^{(\tilde{f},f)}$ inherits from $V_1$ a (smooth) vector space structure is obvious from the following 
considerations: over a point not in $\chi_1^f(Y)$, it coincides with the corresponding fibre of $V_1$ itself, while over $x\in\chi_1^f(Y)$ it coincides with the quotient of $\pi_1^{-1}(x)\subset V_1$ over the 
kernel of $\tilde{f}|_{\pi_1^{-1}(x)}$. For the same reason, the induced map $\tilde{f}_{\sim}$ is linear on each fibre where it is defined, \emph{i.e.}, on $(\pi_1^{(\tilde{f},f)})^{-1}(x)$ with $x\in\chi_1^f(Y)$.
\end{proof}

The following is then an immediate consequence.

\begin{cor}\label{gluing:p-bundles:yields:p-bundle:cor}
The map $\pi_1^{(\tilde{f},f)}:V_1^{\tilde{f}}\to X_1^f$ is a diffeological vector pseudo-bundle, and the pair $(\tilde{f}_{\sim},f_{\sim})$ defines a gluing of it to the pseudo-bundle $\pi_2:V_2\to X_2$.
Furthermore, the pseudo-bundle
$$\pi_1^{(\tilde{f},f)}\cup_{(\tilde{f}_{\sim},f_{\sim})}\pi_2: V_1^{\tilde{f}}\cup_{\tilde{f}_{\sim}}V_2\to X_1^f\cup_{f_{\sim}}X_2$$ is diffeomorphic to the pseudo-bundle
$$\pi_1\cup_{(\tilde{f},f)}\pi_2:V_1\cup_{\tilde{f}}V_2\to X_1\cup_f X_2.$$ In particular, there is a diffeomorphism 
$$C^{\infty}(X_1\cup_f X_2,V_1\cup_{\tilde{f}}V_2)\cong C^{\infty}(X_1^f\cup_{f_{\sim}}X_2,V_1^{\tilde{f}}\cup_{\tilde{f}_{\sim}}V_2).$$
\end{cor}

\begin{proof}
It is evident from the definition of diffeological gluing that $X_1^f\cup_{f_{\sim}}X_2\cong X_1\cup_f X_2$; it remains to notice that the same kind of diffeomorphism between 
$V_1^{\tilde{f}}\cup_{\tilde{f}_{\sim}}V_2$ and $V_1\cup_{\tilde{f}}V_2$ is fibre-to-fibre relative to, respectively, the projections $\pi_1^{(\tilde{f},f)}\cup_{(\tilde{f}_{\sim},f_{\sim})}\pi_2$ and
$\pi_1\cup_{(\tilde{f},f)}\pi_2$.
\end{proof}

\paragraph{The space of sections $C^{\infty}(X_1^f\cup_{f_{\sim}}X_2,V_1^{\tilde{f}}\cup_{\tilde{f}_{\sim}}V_2)$} By the general definition, sections $s_1^f\in C^{\infty}(X_1^f,V_1^{\tilde{f}})$ and 
$s_2\in C^{\infty}(X_2,V_2)$ are compatible if
$$\tilde{f}_{\sim}(s_1^f(y))=s_2(f_{\sim}(y))\mbox{ for all }y\in\pi_1^f(Y).$$ The advantage of considering the reduced pair $(X_1^f,V_1^{\tilde{f}})$ lies in the presentation, resulting from Corollary 
\ref{gluing:p-bundles:yields:p-bundle:cor}, of the pseudo-bundle $\pi_1\cup_{(\tilde{f},f)}\pi_2:V_1\cup_{\tilde{f}}V_2\to X_1\cup_f X_2$ as one obtained by gluing along a pair of diffeomorphisms 
(thus always possible, as long as we assume that both $f$ and $\tilde{f}$ are subductions onto their respective images). The following is a consequence of 
Proposition\ref{space:of:sections:case:of:two:diffeo:prop} and the above corollary.

\begin{prop}\label{space:of:sections:glued:via:reduced:prop}
Assume that $f$ and $\tilde{f}$ are subductions. Then
$$C^{\infty}(X_1\cup_f X_2,V_1\cup_{\tilde{f}}V_2)\cong C^{\infty}(X_1^f,V_1^{\tilde{f}})\times_{comp}C^{\infty}(X_2,V_2),$$ where the compatibility is with respect to the maps $(f_{\sim},\tilde{f}_{\sim})$.
\end{prop}

\subsubsection{The map $\mathcal{S}_1:C_{(f,\tilde{f})}(X_1,V_1)\to C_{(f_{\sim},\tilde{f}_{\sim})}^{\infty}(X_1^f,V_1^{\tilde{f}})$ and its properties}

To make use of Proposition \ref{space:of:sections:glued:via:reduced:prop}, we need to relate the space $C^{\infty}(X_1^f,V_1^{\tilde{f}})$ to the initial space $C_{(f,\tilde{f})}^{\infty}(X_1,V_1)$.
To do so, we consider the map
$$\mathcal{S}_1:C_{(f,\tilde{f})}^{\infty}(X_1,V_1)\to C^{\infty}(X_1^f,V_1^{\tilde{f}})$$ acting by
$$\mathcal{S}_1:C_{(f,\tilde{f})}^{\infty}(X_1,V_1)\ni s_1\mapsto s_1^f\in C^{\infty}(X_1^f,V_1^{\tilde{f}})\mbox{ such that }s_1^f\circ\chi_1^f=\chi_1^{\tilde{f}}\circ s_1.$$

\paragraph{The map $\mathcal{S}_1$ is well-defined} The definition of $\mathcal{S}_1$ that we have given above is an indirect one, so we must check that it is well given.

\begin{lemma}
For every $s_1\in C_{(f,\tilde{f})}^{\infty}(X_1,V_1)$ there exists and is unique $s_1^f\in C^{\infty}(X_1^f,V_1^{\tilde{f}})$ such that $s_1^f\circ\chi_1^f=\chi_1^{\tilde{f}}\circ s_1$.
\end{lemma}

\begin{proof}
The definition of $s_1^f$ is as follows: for any given point $x\in X_1^f$, let $x'\in X_1^f$ be any point (existing by surjectivity of $\chi_1^f$) such that $\chi_1^f(x')=x$; define $s_1^f(x)$ to be 
$s_1^f(x)=\chi_1^{\tilde{f}}(s_1(x'))$. Let us show that the definition is well-posed; let $x''\in X_1$ be another point such that $\chi_1^f(x'')=x$. We need to show that 
$\chi_1^{\tilde{f}}(s_1(x'))=\chi_1^{\tilde{f}}(s_1(x''))$. Since $\chi_1^f(x'')=x=\chi_1^f(x')$ is equivalent to $f(x')=f(x'')$, by $(f,\tilde{f})$-invariance of $s_1$we obtain that $\tilde{f}(s_1(x'))=\tilde{f}(s_1(x''))$. 
This in turn is equivalent to $\chi_1^{\tilde{f}}(s_1(x'))=\chi_1^{\tilde{f}}(s_1(x''))$, therefore $s_1^f$ is well-defined (and is obviously a map that goes $X_1^f\to V_1^{\tilde{f}}$).

The smoothness of $s_1^f$ follows from the expression that defines it. Specifically, if $p^f:U\to X_1^f$ is a plot then, assuming that $U$ is small enough, there exists a plot $p:U\to X_1$ such that 
$p^f=\chi_1^f\circ p$. Therefore $s_1^f\circ p^f=\chi_1^{\tilde{f}}\circ s_1\circ p$. The latter is a plot of $V_1^{\tilde{f}}$, since $s_1$ is smooth as a map $X_1\to V_1$, and $\chi_1^{\tilde{f}}$ is smooth, 
because the diffeology of $V_1^{\tilde{f}}$ is the pushforward of that of $V_1$ by it.
\end{proof}

We thus obtain the following statement.

\begin{cor}
The map $\mathcal{S}_1$ is well-defined as a map $C_{(f,\tilde{f})}^{\infty}(X_1,V_1)\to C^{\infty}(X_1^f,V_1^{\tilde{f}})$.
\end{cor}

\paragraph{The map $\mathcal{S}_1$ is linear} Recall that $C_{(f,\tilde{f})}^{\infty}(X_1,V_1)$ has the structure of a module over the ring of $f$-invariant functions; likewise, 
$C_{(f_{\sim},\tilde{f}_{\sim})}^{\infty}(X_1^f,V_1^{\tilde{f}})$ has the structure of a module over the ring of $f_{\sim}$-invariant functions. The map $\mathcal{S}_1$ respects these
two structures, as the next statement shows.

\begin{thm}\label{map:S-1:is:linear:thm}
The map $\mathcal{S}_1$ is additive, that is, for any two sections $s_1,s_1'\in C_{(f,\tilde{f})}^{\infty}(X_1,V_1)$ we have
$$\mathcal{S}_1(s_1+s_1')=\mathcal{S}_1(s_1)+\mathcal{S}_1(s_1').$$
Furthermore, if $h:X_1\to\matR$ is an $f$-invariant function and $h^f:X_1^f\to\matR$ is defined by $h=h^f\circ\chi_1^f$ then
$$\mathcal{S}_1(hs_1)=h^f\mathcal{S}_1(s_1).$$
\end{thm}

\begin{proof}
Let $s_1,s_1'\in C_{(f,\tilde{f})}^{\infty}(X_1,V_1)$ be two sections. The images $\mathcal{S}_1(s_1)$, $\mathcal{S}_1(s_1')$, and $\mathcal{S}_1(s_1+s_1')$ are defined by the following
identities:
$$\mathcal{S}_1(s_1)\circ\chi_1^f=\chi_1^{\tilde{f}}\circ s_1,\,\,\,\mathcal{S}_1(s_1')\circ\chi_1^f=\chi_1^{\tilde{f}}\circ s_1',\,\,\,\mathcal{S}_1(s_1+s_1')\circ\chi_1^f=\chi_1^{\tilde{f}}\circ(s_1+s_1').$$
We obviously have
\begin{flushleft}
$\mathcal{S}_1(s_1+s_1')\circ\chi_1^f=\chi_1^{\tilde{f}}\circ(s_1+s_1')=\chi_1^{\tilde{f}}\circ s_1+\chi_1^{\tilde{f}}\circ s_1'=$
\end{flushleft}
\begin{flushright}
$=\mathcal{S}_1(s_1)\circ\chi_1^f+\mathcal{S}_1(s_1')\circ\chi_1^f=(\mathcal{S}_1(s_1)+\mathcal{S}_1(s_1'))\circ\chi_1^f$,
\end{flushright}
that is,
$$\mathcal{S}_1(s_1+s_1')\circ\chi_1^f=(\mathcal{S}_1(s_1)+\mathcal{S}_1(s_1'))\circ\chi_1^f.$$ Since $\chi_1^f$ is surjective, we have $\mathcal{S}_1(s_1+s_1')=\mathcal{S}_1(s_1)+\mathcal{S}_1(s_1')$, 
as wanted. As for the second statement of the theorem, using again $\mathcal{S}_1(s_1)\circ\chi_1^f=\chi_1^{\tilde{f}}\circ s_1$, and the linearity of $\tilde{f}$, we obtain
$$\mathcal{S}_1(hs_1)\circ\chi_1^f=\chi_1^{\tilde{f}}\circ(hs_1)=h(\chi_1^{\tilde{f}}\circ s_1)=(h^f\circ\chi_1^f)(\chi_1^{\tilde{f}}\circ s_1)=(h^f\circ\chi_1^f)(\mathcal{S}_1(s_1)\circ\chi_1^f)=
(h^f\mathcal{S}_1(s_1))\circ\chi_1^f.$$ Again by surjectivity of $\chi_1^f$, we obtain that $\mathcal{S}_1(hs_1)=h^f\mathcal{S}_1(s_1)$, which completes the proof.
\end{proof}

\paragraph{The map $\mathcal{S}_1$ is smooth} We finally show that the map $\mathcal{S}_1$ is smooth for the functional diffeologies on $C_{(f,\tilde{f})}^{\infty}(X_1,V_1)$ and
$C^{\infty}(X_1^f,V_1^{\tilde{f}})$.

\begin{thm}
The map $\mathcal{S}_1$ is a smooth map $C_{(f,\tilde{f})}^{\infty}(X_1,V_1)\to C^{\infty}(X_1^f,V_1^{\tilde{f}})$.
\end{thm}

\begin{proof}
Observe that the functional diffeology on $C_{(f,\tilde{f})}^{\infty}(X_1,V_1)$ is the subset diffeology with respect to its inclusion into $C^{\infty}(X_1,V_1)$. Let $q:U\to C_{(f,\tilde{f})}^{\infty}(X_1,V_1)$ be 
a plot; recall that this means that for any plot $p:U'\to X_1$ the map $(u,u')\mapsto q(u)(p(u'))$ is smooth as a map from $U\times U'$ to $V_1$ (that is, it is a plot of $V_1$). Let us consider the composition 
$\mathcal{S}_1\circ q$; we need to show that it is a plot of $C^{\infty}(X_1^f,V_1^{\tilde{f}})$, and so it should satisfy the analogous condition.

Let $p^f:U'\to X_1^f$ be a plot of $X_1^f$; by the definition of the diffeology of the latter, there is a plot $p$ of $X_1$ such that $p^f=\chi_1^f\circ p$. Thus, we have
$$(u,u')\mapsto(\mathcal{S}_1\circ q)(u)(p(u'))=(\mathcal{S}_1(q(u))\circ\chi_1^f)(p(u')=(\chi_1^{\tilde{f}}\circ q(u))(p(u'))=\chi_1^{\tilde{f}}(q(u)(p(u'))),$$ and since $(u,u')\mapsto q(u)(p(u'))$ is a plot of $V_1$ 
by assumption, the resulting map $(u,u')\mapsto\chi_1^{\tilde{f}}(q(u)(p(u')))$ is a plot of $V_1^{\tilde{f}}$ by the definition of its diffeology, whence the claim.
\end{proof}

In what follows we will show that, while $\mathcal{S}_1$ may not be injective, it is always surjective, which allows the space of sections $C^{\infty}(X_1^f,V_1^{\tilde{f}})$ to act as a substitute for the space 
$C_{(f,\tilde{f})}^{\infty}(X_1,V_1)$.

\subsubsection{The quotient pseudo-bundle $\pi_1^{V_1/Ker(\tilde{f})}:V_1(\tilde{f})\to X_1$}

This auxiliary pseudo-bundle allows to consider the surjectivity of $\mathcal{S}_1$; this reasoning is straightforward. Let $\mbox{Ker}(\tilde{f})$ be the sub-bundle of $V_1$ formed by the union of the 
following subspaces in $\pi_1^{-1}(x)$: the subspace $\mbox{ker}(\tilde{f}|_{\pi_1^{-1}(x)})$ if $x\in Y$, and the zero subspace otherwise. It is endowed with the subset diffeology relative to the inclusion 
$\mbox{Ker}(\tilde{f})\subseteq V_1$ and with the restriction $\pi_1|_{\mbox{Ker}(\tilde{f})}=:\pi_1^{ker}$ of $\pi_1$; with respect to these it is a diffeological vector pseudo-bundle (see \cite{pseudobundles}). 
The sub-bundle thus obtained is called the \textbf{kernel of $\tilde{f}$}.

Consider the corresponding quotient pseudo-bundle with the total space $V_1/\mbox{Ker}(\tilde{f})=:V_1(\tilde{f})$, the base space $X_1$, and the induced pseudo-bundle projection denoted by 
$\pi_1^{V_1/Ker(\tilde{f})}:V_1(\tilde{f})\to X_1$. The usual quotient projection $V_1\to V_1/\mbox{Ker}(\tilde{f})=V_1(\tilde{f})$ will be denoted by $\chi_1^{V_1(\tilde{f})}$. This quotient projection covers 
the identity map on $X_1$, that is,
$$\pi_1=\pi_1^{V_1/Ker(\tilde{f})}\circ\chi_1^{V_1(\tilde{f})}.$$

\begin{lemma}
There is a smooth surjective pseudo-bundle map $\chi_1^0:V_1(\tilde{f})\to V_1^{\tilde{f}}$ covering the map $\chi_1^f:X_1\to X_1^f$.
\end{lemma}

\begin{proof}
This follows from the construction of $V_1(\tilde{f})$ and that of $V_1^{\tilde{f}}$. Indeed, $V_1^{\tilde{f}}$ is the quotient of $V_1(\tilde{f})$ by the following equivalence relation. Let $\tilde{f}_0$ be the 
pushforward of $\tilde{f}$ to the quotient $V_1(\tilde{f})$. The space $(V_1(\tilde{f}))^{\tilde{f}_0}$, defined as the quotient of $V_1(\tilde{f})$ by the equivalence relation 
$v_1\sim^0 v_2\Leftrightarrow\tilde{f}_0(v_1)=\tilde{f}_0(v_2)$, is then precisely the space $V_1^{\tilde{f}}$.
\end{proof}

\begin{rem}
The pseudo-bundle map $(\chi_1^{\tilde{f}},\chi_1^f)$ filters through the pseudo-bundle map $(\chi_1^{V_1(\tilde{f})},\mbox{Id}_{X_1})$, that is, we have
$$\chi_1^{\tilde{f}}=\chi_1^0\circ\chi_1^{V_1(\tilde{f})}.$$ Furthermore, there is an induced gluing of $\pi_1^{V_1/Ker(\tilde{f})}:V_1(\tilde{f})\to X_1$ to $\pi_2:V_2\to X_2$
along the maps $(\tilde{f}_0,f)$ that yields the pseudo-bundle
$$\pi_1^{V_1/Ker(\tilde{f})}\cup_{(\tilde{f}_0,f)}\pi_2:V_1(\tilde{f})\cup_{\tilde{f}_0}V_2\to X_1\cup_f X_2.$$
\end{rem}

\subsubsection{$\mathcal{S}_1$ may not be injective}

We now indicate the reason why $\mathcal{S}_1$ may not be injective, although for reasons of brevity we do not provide a complete treatment of any specific example (which is easy to find anyway).

\begin{oss}
Let $s_1$ and $s_1'$ be two sections in $C^{\infty}(X_1,V_1)$ such that $\chi_1^{\tilde{f}}\circ s_1=\chi_1^{\tilde{f}}\circ s_1'$; suppose that there exists $x\in X_1$ such that $s_1(x)\neq s_1'(x)$. Since 
by assumption $\chi_1^{\tilde{f}}(s_1(x))=\chi_1^{\tilde{f}}(s_1'(x))$, and $\chi_1^{\tilde{f}}$ is injective outside of $\pi_1^{-1}(Y)$, we conclude that $x\in Y$, and that $s_1(x)-s_1'(x)$ belongs to the fibre 
at $x$ of $\mbox{Ker}(\tilde{f})$. This is easily seen to be a \emph{vice versa}.
\end{oss}

A concrete example could be obtained by assuming that $\mbox{Ker}(\tilde{f})$ splits off as a smooth direct summand in the pseudo-bundle $V_1$ and is such that there exists a smooth non-zero section 
$s_0:X\to\mbox{Ker}(\tilde{f})$. These assumptions suffice for $\mathcal{S}_1$ to be non-injective. More precisely, for any section $s_1\in C_{(f,\tilde{f})}^{\infty}(X_1,V_1)$ we have that 
$s_1+s_0\in C_{(f,\tilde{f})}^{\infty}(X_1,V_1)$.

Indeed, if $y,y'\in Y\subset X_1$ are such that $f(y)=f(y')$, then by assumption and by linearity of $\tilde{f}$ 
$$\tilde{f}(s_1(f(y)))=\tilde{f}(s_1(f(y')))\Rightarrow\tilde{f}(s_1(f(y))+s_0(f(y)))=\tilde{f}(s_1(f(y))).$$ Similarly,
$$\tilde{f}(s_1(f(y'))+s_0(f(y')))=\tilde{f}(s_1(f(y')))\Leftrightarrow\tilde{f}((s_1+s_0)(f(y)))=\tilde{f}((s_1+s_0)(f(y')));$$ in particular, $\mathcal{S}_1(s_1+s_0)$ is well-defined. It remains to observe that
by Theorem \ref{map:S-1:is:linear:thm}
$$\mathcal{S}_1(s_1+s_0)=\mathcal{S}_1(s_1)+\mathcal{S}_1(s_0),$$ and since $\mathcal{S}_1(s_0)$ is the zero section, this is equal to $\mathcal{S}_1(s_1)$. Since $s_1+s_0\neq s_1$ by
the choice of $s_0$, we see that $\mathcal{S}_1$ is not injective.

\subsubsection{Surjectivity of $\mathcal{S}_1$: the case of the trivial $\mbox{Ker}(\tilde{f})$}

We treat the case of the trivial $\mbox{Ker}(\tilde{f})$ separately, since for obvious reasons it is possible to obtain stronger statements in this case. Indeed, the assumption that $\mbox{Ker}(\tilde{f})$ is trivial 
implies that $V_1(\tilde{f})=V_1$, and allows to define, for any given section $s\in C^{\infty}(X_1^f,V_1^{\tilde{f}})$, its pullback via the map$\mathcal{S}_1$ to a well-defined and unique section $X_1\to V_1$. 
This pullback, denoted by $\mathcal{S}_1^{-1}(s)$, is given by the following formula:
$$\mathcal{S}_1^{-1}(s)(x)=\left\{\begin{array}{ll}
(\chi_1^{\tilde{f}})^{-1}(s(\chi_1^f(x))) & \mbox{for }x\in X_1\setminus Y \\
(\chi_1^{\tilde{f}}|_{\chi_1^{-1}(x)})^{-1}(s(\chi_1^f(x))) & \mbox{for }x\in Y.
\end{array}\right.$$
Since under the present assumption the restriction of $\tilde{f}$ on each individual fibre in its domain is injective, the restriction of $\chi_1^{\tilde{f}}$ on any fibre in $V_1$ is injective as well. It is also 
obvious that the map $\mathcal{S}_1^{-1}(s)$ thus obtained is $(f,\tilde{f})$-invariant. We need to verify is that it is smooth as a map $X_1\to V_1$.

\begin{lemma}\label{inverse:of:S-1:well-defined:lem}
The map $\mathcal{S}_1^{-1}(s):X_1\to V_1$ is smooth for every $s\in C^{\infty}(X_1^f,V_1^{\tilde{f}})$.
\end{lemma}

\begin{proof}
Let $p:U\to X_1$ be a plot of $X_1$; we need to show that $u\mapsto\mathcal{S}_1^{-1}(s)(p(u))$ is a plot of $V_1$. By definition of a pushforward diffeology, this is equivalent, for $U$ small enough, 
to $u\mapsto\chi_1^{\tilde{f}}(\mathcal{S}_1^{-1}(s)(p(u)))$ being a plot of $V_1^{\tilde{f}}$. By an easy calculation we obtain
$$\chi_1^{\tilde{f}}(\mathcal{S}_1^{-1}(s)(p(u)))=s(\chi_1^f(p(u))).$$ Since $\chi_1^f$ is smooth by construction, $\chi_1^f\circ p$ is a plot of $X_1^f$, and since $s$ is smooth by assumption,
$s\circ\chi_1^f\circ p$ is a plot of $V_1^{\tilde{f}}$, whence the claim.
\end{proof}

Lemma \ref{inverse:of:S-1:well-defined:lem} yields a well-defined inverse map 
$$\mathcal{S}_1^{-1}:C^{\infty}(X_1^f,V_1^{\tilde{f}})\to C_{(f,\tilde{f})}^{\infty}(X_1,V_1).$$ Moreover, we have the following statement.

\begin{thm}\label{when:mathcal:S:is:smoothly:invertible:thm}
Let $\pi_1:V_1\to X_1$ and $\pi_2:V_2\to X_2$ be two diffeological vector pseudo-bundles, and let $(\tilde{f},f)$ be a gluing between them such that $\tilde{f}$ is injective on each fibre in its domain of 
definition. Then $\mathcal{S}_1^{-1}$ is smooth as a map $C^{\infty}(X_1^f,V_1^{\tilde{f}})\to C_{(f,\tilde{f})}^{\infty}(X_1,V_1)$.
\end{thm}

\begin{proof}
Let $q:U\to C^{\infty}(X_1^f,V_1^{\tilde{f}})$ be a plot of $C^{\infty}(X_1^f,V_1^{\tilde{f}})$; recall that, as for any functional diffeology, this means that for any plot $p^f:U'\to X_1^f$ of $X_1^f$
the corresponding evaluation map $(u,u')\mapsto q(u)(p^f(u'))$ is a plot of $V_1^{\tilde{f}}$. Let us show that the evaluation map corresponding to $\mathcal{S}_1^{-1}\circ q$ is a plot of $V_1$.

Let $p:U'\to X_1$ be a plot of $X_1$. As in the previous proof, up to restricting $U$ and $U'$ as necessary, it would be sufficient to prove that
$(u,u')\mapsto\chi_1^{\tilde{f}}((\mathcal{S}_1^{-1}\circ q)(u)(p(u')))$ is a plot of $V_1^{\tilde{f}}$. By definition of $\mathcal{S}_1^{-1}$ we have
$$\chi_1^{\tilde{f}}((\mathcal{S}_1^{-1}\circ q)(u)(p(u')))=q(u)(\chi_1^f(p(u'))).$$ Since $\chi_1^f\circ p$ is a plot of $X_1^f$ by construction, the resulting map is a plot of $V_1^{\tilde{f}}$
by the assumption on $q$, whence the claim.
\end{proof}

The following is now an obvious conclusion.

\begin{cor}\label{quotient:to:reduced:map:surjective:cor}
Under the assumptions of Theorem \ref{when:mathcal:S:is:smoothly:invertible:thm}, the spaces of sections $C_{(f,\tilde{f})}^{\infty}(X_1,V_1)$ and $C^{\infty}(X_1^f,V_1^{\tilde{f}})$ are diffeomorphic.
\end{cor}

\subsubsection{Surjectivity of $\mathcal{S}_1$ in the case when $\mbox{Ker}(\tilde{f})$ is non-trivial}

We first assume for simplicity that $f$ is injective, so $X_1^f=X_1$ and $V_1(\tilde{f})=V_1^{\tilde{f}}$. Notice also that under this assumption $\mathcal{S}_1$ is determined by the simpler condition
$\mathcal{S}_1(s')=\chi_1^{\tilde{f}}\circ s'$.

\begin{prop}\label{quotient:map:surjective:sections:prop}
Let $(\tilde{f},f)$ be such that $f$ is injective. Then for every smooth section $s:X_1\to V_1(\tilde{f})$ there exists a smooth $(f,\tilde{f})$-invariant section $s':X_1\to V_1$ such that $\mathcal{S}_1(s')=s$.
\end{prop}

\begin{proof}
Since by assumption $V_1$ has finite-dimensional fibres only, we can choose an arbitrary direct sum decomposition $V_1=V_1^0\oplus\mbox{Ker}(\tilde{f})$. The direct sum complement
$V_1^0$ thus chosen is also a sub-bundle, but if the decomposition is not smooth then the diffeology of $V_1$ is coarser than the respective direct sum diffeology. Also, having fixed such a
decomposition, for every section $s:X_1\to V_1(\tilde{f})$ there is a well-defined pullback of it to a (non-smooth \emph{a priori}) section $X_1\to V_1$.

Since $V_1(\tilde{f})=V_1/\mbox{Ker}(\tilde{f})$, we can write its elements as cosets $v+\mbox{Ker}(\tilde{f})$. The map $\chi_1^{\tilde{f}}$ then has form $\chi_1^{\tilde{f}}(v)=v+\mbox{Ker}(\tilde{f})$,
and every plot of $V_1(\tilde{f})$ has form $\chi_1^{\tilde{f}}\circ p$ for some plot $p$ of $V_1$. Now, if $s:X_1\to V_1(\tilde{f})$ is a smooth section, then for any given $x\in X_1$ we can
denote by $t(s)(x)$ the unique element of $V_1^0$ contained in the coset $s(x)$. The map $t(s)$ thus defined is a section $X_1\to V_1$.

To show that $t(s)$ is smooth as a map $X_1\to V_1$, let $q:U\to X_1$ be a plot of $X_1$. We need to show that $u\mapsto t(s)(q(u))$ is a plot of $V_1$. This is equivalent to showing that
there exists a sub-domain $U'$ of $U$ such that on this sub-domain $u\mapsto\chi_1^{\tilde{f}}(t(s)(q(u)))$ is a plot of $V_1(\tilde{f})$. But we have by construction that
$\chi_1^{\tilde{f}}(t(s)(q(u)))=s(q(u))$ on the whole $U$. Since by assumption $s$ is smooth as a map $X_1\to V_1(\tilde{f})$, we have that $u\mapsto s(q(u))$ is a plot of $V_1(\tilde{f})$.
The map $t(s)$ is thus the section $s'$ we were looking for; in particular, it is clearly $(f,\tilde{f})$-invariant.
\end{proof}

\begin{example}
Let $V_1=\matR\times\matR^2$, with the first factor carrying the standard diffeology and the second, the vector space diffeology generated by $u\mapsto|u|(e_y+e_z)$; let $X_1$ be the standard
$\matR$ identified with the first factor, so the second factor is the fibre. Let $\tilde{f}$ be defined over the whole $X_1$ (so on the entire $V_1$), and let it act by $(x,y,z)\mapsto(x,0,z)$; we may
assume it to take values in some $V_2=\matR\times\matR$, where the first factor is the standard $\matR$ identified with the corresponding base space and the second is $\matR$ with the vector
space diffeology generated by $u\mapsto|u|e_z$. Thus, $\tilde{f}$ is smooth, and $V_1(\tilde{f})$ can actually be identified with $V_2$. It is convenient to represent both of them by the subset
$\{(x,0,z)\}$ of $\matR^3$.

Observe that every section $X_1\to V_1(\tilde{f})$ is a linear combination with coefficients that are usual smooth functions in $x$ of sections of form $x\mapsto(x,0,|g(x)|)$ (where again $g$ is a
usual smooth function). It is then obvious that every such map lifts to the section $X_1\to V_1$ that is given by $x\mapsto(x,|g(x)|,|g(x)|)$.
\end{example}

Corollary \ref{quotient:to:reduced:map:surjective:cor} and Proposition \ref{quotient:map:surjective:sections:prop} allow to show that $\mathcal{S}_1$ is always surjective.

\begin{thm}\label{mathcal:S:is:surjective:thm}
Let $\pi_1:V_1\to X_1$ and $\pi_2:V_2\to X_2$ be two diffeological vector pseudo-bundles, and let $(\tilde{f},f)$ be a gluing of the former to the latter such that $f$ and $\tilde{f}$ are subductions onto their
respective images. Then the map $\mathcal{S}_1$ is surjective as a map $C_{(f,\tilde{f})}^{\infty}(X_1,V_1)\to C^{\infty}(X_1^f,V_1^{\tilde{f}})$.
\end{thm}

\begin{proof}
Recall that the pseudo-bundle map $(\chi_1^{\tilde{f}}:V_1\to V_1^{\tilde{f}},\chi_1^f:X_1\to X_1^f)$ filters through the pseudo-bundle maps
$(\chi_1^{V_1(\tilde{f})}:V_1\to V_1(\tilde{f}),\mbox{Id}_{X_1}:X_1\to X_1)$ and $(\chi_1^0:V_1(\tilde{f})\to V_1^{\tilde{f}},\chi_1^f:X_1\to X_1^f)$. Accordingly, $\mathcal{S}_1$ decomposes into
the following composition of maps.

Let $\mathcal{S}_1^{V_1(\tilde{f})}:C_{(f,\tilde{f})}^{\infty}(X_1,V_1)\to C^{\infty}(X_1,V_1(\tilde{f}))$ be the map defined by
$$\mathcal{S}_1^{V_1(\tilde{f})}(s)=\chi_1^{V_1(\tilde{f})}\circ s$$ (it coincides with $\mathcal{S}_1$ if $f$ and $\tilde{f}$ are such that $V_1(\tilde{f})=V_1^{\tilde{f}}$). Let
$\mathcal{S}_1^0:C_{(f,\tilde{f}_0)}^{\infty}(X_1,V_1(\tilde{f}))\to C^{\infty}(X_1^f,V_1^{\tilde{f}})$ be the map defined by
$$\mathcal{S}_1^0(s)\circ\chi_1^f=\chi_1^0\circ s.$$ We claim, first of all, that
$$\mathcal{S}_1=\mathcal{S}_1^0\circ\mathcal{S}_1^{V_1(\tilde{f})}.$$

Indeed,
$$\left(\mathcal{S}_1^0\circ\mathcal{S}_1^{V_1(\tilde{f})}\right)(s)=\mathcal{S}_1^0\left(\chi_1^{V_1(\tilde{f})}\circ s\right),$$ and the latter satisfies the identity
$$\mathcal{S}_1^0\left(\chi_1^{V_1(\tilde{f})}\circ s\right)\circ\chi_1^f=\chi_1^0\circ\chi_1^{V_1(\tilde{f})}\circ s=\chi_1^{\tilde{f}}\circ s$$ by Lemma 4.22. Since
$\chi_1^{\tilde{f}}\circ s=\mathcal{S}_1(s)\circ\chi_1^f$, we get that $\mathcal{S}_1(s)=(\mathcal{S}_1^0\circ\mathcal{S}_1^{V_1(\tilde{f})})(s)$ for any $(f,\tilde{f})$-invariant section
$s:X_1\to V_1$.

Now, by Corollary \ref{quotient:to:reduced:map:surjective:cor} the map $\mathcal{S}_1^0$ is a diffeomorphism between $C_{(f,\tilde{f}_0)}(X_1,V_1(\tilde{f}))$ and $C^{\infty}(X_1^f,V_1^{\tilde{f}})$.
It thus suffices to show that $\mathcal{S}_1^{V_1(\tilde{f})}$ maps $C_{(f,\tilde{f})}^{\infty}(X_1,V_1)$ onto $C_{(f,\tilde{f}_0)}^{\infty}(X_1,V_1(\tilde{f}))$. This is obtained by first applying
Proposition \ref{quotient:map:surjective:sections:prop} where instead of $f$ we consider $\mbox{Id}_{X_1}$ and instead of $\tilde{f}$, the quotient map $\chi_1^{V_1(\tilde{f})}$. The proposition then
guarantees that every section $X_1\to V_1(\tilde{f})$ pulls back to a $(\mbox{Id}_{X_1},\chi_1^{V_1(\tilde{f})})$-invariant section $X_1\to V_1$. We thus need to check that any
$(f,\tilde{f}_0)$-invariant section admits a pullback that is $(f,\tilde{f})$-invariant; and this easily follows from $\tilde{f}=\tilde{f}_0\circ\chi_1^{V_1(\tilde{f})}$, \emph{i.e.},
from the very definition of $\tilde{f}_0$. Thus, as a map $C_{(f,\tilde{f})}^{\infty}(X_1,V_1)\to C_{(f,\tilde{f}_0)}^{\infty}(X_1,V_1(\tilde{f}))$, the map $\mathcal{S}_1^{V_1(\tilde{f})}$ is onto,
which completes the proof.
\end{proof}

\subsubsection{$\mathcal{S}_1$ is a subduction}

We have just seen (Theorem \ref{mathcal:S:is:surjective:thm}) that if $\tilde{f}$ and $f$ are subductions then $\mathcal{S}_1$ is surjective. We now show that a stronger statement is true: under the same 
assumption, $\mathcal{S}_1$ is a subduction itself. 

\begin{thm}\label{mathcal:S_1:is:subduction:thm}
Let $\pi_1:V_1\to X_1$ and $\pi_2:V_2\to X_2$ be two diffeological vector pseudo-bundles, and let $(\tilde{f},f)$ be a gluing of the former to that latter such that both $\tilde{f}$ and $f$ are subsections 
onto their images. Then the map $\mathcal{S}_1$ is a subduction of $C_{(f,\tilde{f})}^{\infty}(X_1,V_1)$ onto $C_{(f_{\sim},\tilde{f}_{\sim})}^{\infty}(X_1^f,V_1^{\tilde{f}})$.
\end{thm}

\begin{proof}
We need to show that every plot $q^{f,\tilde{f}}$ of $C_{(f_{\sim},\tilde{f}_{\sim})}^{\infty}(X_1^f,V_1^{\tilde{f}})$ locally has form $\mathcal{S}_1\circ q$ for some plot $q$ of
$C_{(f,\tilde{f})}^{\infty}(X_1,V_1)$. Thus, let $q^{f,\tilde{f}}:U\to C_{(f_{\sim},\tilde{f}_{\sim})}^{\infty}(X_1^f,V_1^{\tilde{f}})$ (we will assume that $U$ is small enough, as needed);
this means that for any plot $p^f:U'\to X_1^f$ of $X_1^f$ the usual evaluation map $(u,u')\mapsto q^{f,\tilde{f}}(u)(p^f(u'))$ is a plot of $V_1^{\tilde{f}}$. Now we also assume that
$U'$ is small enough so that $p^f=\chi_1^f\circ p$ for some plot $p$ of $X_1$.

As shown in the proof of Proposition \ref{quotient:map:surjective:sections:prop}, the map $\mathcal{S}_1$ admits a right inverse, depending on the choice of a decomposition
of $V$ into a direct sum with $\mbox{Ker}(\tilde{f})$. Let $(\mathcal{S}_1)^{-1}$ be any fixed choice of a right inverse; define a map $q:U\to C_{(f,\tilde{f})}^{\infty}(X_1,V_1)$
by setting $q=(\mathcal{S}_1)^{-1}\circ q^{f,\tilde{f}}\circ\chi_1^f$. By the usual definition, this is a plot if, up to further restricting $U$, we have that, for any given plot $p:U'\to X_1$
of $X_1$, the following is a plot of $V_1$:
$$(u,u')\mapsto(\mathcal{S}_1)^{-1}(q^{f,\tilde{f}}(u))(p(u')).$$ Now, if we assume $U$ and $U'$ to be small enough, this is a plot of $V_1$ if and only if the following is a plot of $V_1^{\tilde{f}}$:
$$(u,u')\mapsto\chi_1^{\tilde{f}}((\mathcal{S}_1)^{-1}(q^{f,\tilde{f}}(u))(p(u'))).$$

Recalling now the definition of $(\mathcal{S}_1)^{-1}$, we get that
$$\chi_1^{\tilde{f}}((\mathcal{S}_1)^{-1}(q^{f,\tilde{f}}(u))(p(u')))=q^{f,\tilde{f}}(u)(\chi_1^f(p(u'))),$$ which is the value of the evaluation of $q^{f,\tilde{f}}(u)$ on the plot $\chi_1^f\circ p$ of
$X_1^f$. Therefore it is a plot of $V_1^{\tilde{f}}$, due to $q^{f,\tilde{f}}$ being a plot of $C_{(f_{\sim},\tilde{f}_{\sim})}^{\infty}(X_1^f,V_1^{\tilde{f}})$, so we conclude that $q$ is
indeed a plot of $C_{(f,\tilde{f})}^{\infty}(X_1,V_1)$. Since $q^{f,\tilde{f}}=\mathcal{S}_1\circ q$ by construction, and it was arbitrarily chosen, we obtain the claim.
\end{proof}

\subsubsection{$\mathcal{S}_1$ preserves compatibility}

The only item that is still lacking for relating the pseudo-bundle $\pi_1\cup_{(\tilde{f},f)}\pi_2:V_1\cup_{\tilde{f}}V_2\to X_1\cup_f X_2$ to its reduced version 
$\pi_1^{(\tilde{f},f)}\cup_{(\tilde{f}_{\sim},f_{\sim})}\pi_2:V_1^{\tilde{f}}\cup_{\tilde{f}_{\sim}}V_2\to X_1\cup_{f_{\sim}}X_2$ is a description of the interaction of the map $\mathcal{S}_1$ with the two 
compatibility conditions (one relative to $(\tilde{f},f)$ and the other to $(\tilde{f}_{\sim},f_{\sim})$). We provide it in this section.

\begin{prop}\label{mathcal:S_1:preserves:compatibility:part1:prop}
For a given gluing $(\tilde{f},f)$ of a pseudo-bundle $\pi_1:V_1\to X_2$ to another pseudo-bundle $\pi_2:V_2\to X_2$, assume that both $\tilde{f}$ and $f$ are subductions, and let $s_i\in C^{\infty}(X_i,V_i)$ 
for $i=1,2$. If $s_1$ and $s_2$ are $(f,\tilde{f})$-compatible then $\mathcal{S}_1(s_1)$ and $s_2$ are $(f_{\sim},\tilde{f}_{\sim})$-compatible.
\end{prop}

Recall that $s_1$ being $(f,\tilde{f})$-compatible with some $s_2$ implies it being $(f,\tilde{f})$-invariant, so the expression $\mathcal{S}_1(s_1)$ makes sense.

\begin{proof}
The $(f,\tilde{f})$-compatibility of $s_1$ and $s_2$ means precisely that for all $y\in Y$ we have $\tilde{f}(s_1(y))=s_2(f(y))$; we need to show that
$\tilde{f}_{\sim}(\mathcal{S}_1(s_1)(\chi_1^f(y)))=s_2(f_{\sim}(\chi_1^f(y)))$. By definition $f_{\sim}(\chi_1^f(y))=f(y)$ and $\mathcal{S}_1(s_1)\circ\chi_1^f=\chi_1^{\tilde{f}}\circ s_1$,
so the desired condition is equivalent to $\tilde{f}_{\sim}(\chi_1^{\tilde{f}}(s_1(y)))=s_2(f(y))$. It remains to notice that $\tilde{f}_{\sim}(\chi_1^{\tilde{f}}(s_1(y)))=\tilde{f}(s_1(y))$ by
definition of $\tilde{f}_{\sim}$, so the $(f_{\sim},\tilde{f}_{\sim})$-compatibility does follow from the $(f,\tilde{f})$-compatibility of $s_1$ and $s_2$.
\end{proof}

The inverse of Proposition \ref{mathcal:S_1:preserves:compatibility:part1:prop} is true as well.

\begin{prop}\label{mathcal:S_1:preserves:compatibility:part2:prop}
Let $s_1\in C_{(f,\tilde{f})}^{\infty}(X_1,V_1)$ and $s_2\in C^{\infty}(X_2,V_2)$ be two sections such that $\mathcal{S}_1(s_1)$ and $s_2$ are $(f_{\sim},\tilde{f}_{\sim})$-compatible.
Then $s_1$ and $s_2$ are $(f,\tilde{f})$-compatible.
\end{prop}

\begin{proof}
The proof is the same as the previous one, just going in the opposite direction. Let $\chi_1^f(y)$ be a point in the domain of $f_{\sim}$; the assumption of $(f_{\sim},\tilde{f}_{\sim})$-compatibility
means precisely that
$$\tilde{f}_{\sim}(\mathcal{S}_1(s_1)(\chi_1^f(y)))=s_2(f_{\sim}(\chi_1^f(y))).$$ Recall that $f_{\sim}\circ\chi_1^f=f$ by definition, so the right-hand side coincides with $s_2(f(y))$. Since
$\mathcal{S}_1$ is defined by the identity $\mathcal{S}_1(s_1)\circ\chi_1^f=\chi_1^{\tilde{f}}\circ s_1$, the left-hand side becomes $\tilde{f}_{\sim}(\chi_1^{\tilde{f}}(s_1(y)))$. Since
furthermore $\tilde{f}_{\sim}\circ\chi_1^{\tilde{f}}=\tilde{f}$ (by the definition of the map $\tilde{f}_{\sim}$), the left-hand side is then equal to $\tilde{f}_{\sim}(s_1(y))$. We thus have
$$\tilde{f}_{\sim}(s_1(y))=\tilde{f}_{\sim}(\mathcal{S}_1(s_1)(\chi_1^f(y)))=s_2(f_{\sim}(\chi_1^f(y)))=s_2(f(y)),$$ \emph{i.e.}, that $s_1$ and $s_2$ are $(f,\tilde{f})$-compatible.
\end{proof}

Putting the two propositions together, we obtain the following result.

\begin{cor}\label{mathcal:S_1:preserves:compatibility:cor}
Suppose that both $\tilde{f}$ and $f$ are subductions. Then $(\mathcal{S}_1,\mbox{Id}_{C^{\infty}(X_2,V_2)})$ is well-defined and surjective as a map
$C_{(f,\tilde{f})}^{\infty}(X_1,V_1)\times_{comp}C^{\infty}(X_2,V_2)\to C_{(f_{\sim},\tilde{f}_{\sim})}^{\infty}(X_1^f,V_1^{\tilde{f}})\times_{comp}C^{\infty}(X_2,V_2)$.
\end{cor}

\subsubsection{The space $C^{\infty}(X_1\cup_f X_2,V_1\cup_{\tilde{f}}V_2)$ is a smooth surjective image of $C_{(f,\tilde{f})}^{\infty}(X_1,V_1)\times_{comp}C^{\infty}(X_2,V_2)
\subseteq C^{\infty}(X_1,V_1)\times C^{\infty}(X_2,V_2)$}

We now collect the results of the current section into the final statement, which is as follows.

\begin{thm}\label{subduction:product-comp:onto:glued:thm}
Let $\pi_1:V_1\to X_1$ and $\pi_2:V_2\to X_2$ be two diffeological vector pseudo-bundles, and let $(\tilde{f},f)$ be a gluing of the former pseudo-bundle to the latter, such that both $\tilde{f}$ and $f$ are 
subductions onto their respective images. The map $\mathcal{S}$ is a subduction of $C_{(f,\tilde{f})}^{\infty}(X_1,V_1)\times_{comp}C^{\infty}(X_2,V_2)$ onto $C^{\infty}(X_1\cup_f X_2,V_1\cup_{\tilde{f}}V_2)$.
\end{thm}

\begin{proof}
It suffices to recall the diffeomorphism $C^{\infty}(X_1^f\cup_{f_{\sim}}X_2,V_1^{\tilde{f}}\cup_{\tilde{f}_{\sim}}V_2)\cong C^{\infty}(X_1\cup_f X_2,V_1\cup_{\tilde{f}}V_2)$ of Proposition 
\ref{space:of:sections:glued:via:reduced:prop}, which for the moment we denote by $\tilde{F}$. By Theorem \ref{mathcal:S:is:smooth:thm} we have two versions of the map $\mathcal{S}$, one for the 
original pseudo-bundle, and one for its restricted version: 
$$\mathcal{S}:C_{(f,\tilde{f})}^{\infty}(X_1,V_1)\times_{comp}C^{\infty}(X_2,V_2)\to C^{\infty}(X_1\cup_f X_2,V_1\cup_{\tilde{f}}V_2)\,\,\,\mbox{ and}$$ 
$$\mathcal{S}^{(f,\tilde{f})}:C^{\infty}(X_1^f,V_1^{\tilde{f}})\times_{comp}C^{\infty}(X_2,V_2)\to C^{\infty}(X_1^f\cup_{f_{\sim}}X_2,V_1^{\tilde{f}}\cup_{\tilde{f}_{\sim}}V_2),$$ that by the same theorem are 
smooth. By Corollary \ref{mathcal:S_1:preserves:compatibility:cor} there is a well-defined and factor-to-factor map
$$(\mathcal{S}_1,\mbox{Id}_{C^{\infty}(X_2,V_2)}):C_{(f,\tilde{f})}^{\infty}(X_1,V_1)\times_{comp}C^{\infty}(X_2,V_2)\to C^{\infty}(X_1^f,V_1^{\tilde{f}})\times_{comp}C^{\infty}(X_2,V_2),$$ 
\emph{i.e.}, one that acts as $\mathcal{S}_1$ on the first factor and as the identity map on the second factor. Observing now that 
$$\mathcal{S}=\tilde{F}\circ\mathcal{S}^{(f,\tilde{f})}\circ(\mathcal{S}_1,\mbox{Id}_{C^{\infty}(X_2,V_2)}),$$ it follows from Propostion \ref{space:of:sections:case:of:two:diffeo:prop}, implying that 
$\mathcal{S}^{(f,\tilde{f})}$ is a diffeomorphism, and Theorem \ref{mathcal:S_1:is:subduction:thm} that $\mathcal{S}$ is a subduction, which completes the proof.
\end{proof}

\section{Diffeological connections: the {\tt verbatim} extension}

One can define a diffeological connection by the minimal possible extension of the standard definition of a Riemannian connection. The resulting notion is then as follows.

\begin{defn}
Let $\pi:V\to X$ be a finite-dimensional diffeological vector pseudo-bundle, and let $C^{\infty}(X,V)$ be the space of its smooth sections. A \textbf{connection} on this pseudo-bundle
is a smooth linear operator
$$\nabla:C^{\infty}(X,V)\to C^{\infty}(X,\Lambda^1(X)\otimes V),$$ which satisfies the Leibntz rule, that is, for every function $f\in C^{\infty}(X,\matR)$ and for every section
$s\in C^{\infty}(X,V)$ we have $$\nabla(fs)=df\otimes s+f\nabla s.$$
\end{defn}

We need to explain first of all why this definition is well-posed. The meaning of the question is as follows. Although, as already mentioned, the differentials of functions are well-defined in the diffeological 
context, they are elements of $\Omega^1(X)$, while for the statement of the Leibniz rule we need them to be sections of $\Lambda^1(X)$. For this reason the meaning of $df$ is one of the section 
given by 
$$df:x\mapsto\pi^{\Omega,\Lambda}(x,df)$$ (we keep the same symbol for both $df$ an element of $\Omega^1(X)$ and $df$ a section of $\Lambda^1(X)$). Having specified this, the definition 
is well-posed.

\subsection{An example for a nonstandard pseudo-bundle}\label{diffeological:connection:ex:sect}

Let us describe first of all an example of a diffeological connection that is not a standard connection on a smooth manifold.

\paragraph{The pseudo-bundle and its gluing presentation} We consider the pseudo-bundle $\pi:V\to X$, where $X$ and $V$ are the following subsets of $\matR^3$:
$$X=\{xy=0,z=0\},\,\,\,\,V=\{xy=0\},$$ and $\pi$ is the restriction to $V$ of the standard projection of $\matR^3$ onto $xy$-coordinate plane. Each fibre $\pi^{-1}(x,y,0)$ of $V$
is endowed with the vector space structure of the usual $\matR$ relative to the third coordinate (keeping the first two fixed):
$$(x,y,z_1)+(x,y,z_2)=(x,y,z_1+z_2),\,\,\,\,\lambda(x,y,z)=(x,y,\lambda z)\mbox{ for }\lambda\in\matR.$$ The diffeologies on $V$ and $X$ are gluing diffeologies coming from their
presentations as
$$X=X_1\cup_f X_2,\,\,X_1=\{y=z=0\},\,\,X_2=\{x=z=0\}\,\,\,\,\mbox{ and }\,\,\,\,V=V_1\cup_{\tilde{f}}V_2,\,\,V_1=\{y=0\},\,\,V_2=\{x=0\},$$ where the gluing maps $f$ and $\tilde{f}$ are the
restrictions of the identity map $\matR^3\to\matR^3$ to their domains of definition; these domains of definition are, the origin $\{(0,0,0)\}$ for $f$, and the $z$-axis $\{(0,0,z)\}$ for
$\tilde{f}$. The four spaces $X_1,X_2,V_1,V_2$ carry the subset diffeology relative to their inclusions into $\matR^3$, and the gluing diffeologies on $X$ and $V$ correspond to those;
notice that these gluing diffeologies are strictly weaker than their subset diffeologies relative to $\matR^3$ (see \cite{watts}). We denote the restrictions of $\pi$ to $V_1$ and to $V_2$
by $\pi_1$ and $\pi_2$ respectively.

\paragraph{Pseudo-metrics on $\pi:V\to X$} We obtain a pseudo-metric on $\pi:V\to X$ by gluing two compatible pseudo-metrics on $\pi_1:V_1\to X_1$ and $\pi_2:V_2\to X_2$
respectively. We denote them by $g_1$ and $g_2$ respectively and define them to be $g_1(x,0,0)=h_1(x)dz^2$ and $g_2(0,y,0)=h_2(y)dz^2$, where $h_1,h_2:\matR\to\matR$ are
usual smooth functions; they obviously need to be everywhere positive. The compatibility condition for them takes form $h_1(0)=h_2(0)$. Assuming this, we obtain a pseudo-metric
$\tilde{g}$ on $V$ defined by
$$\tilde{g}(x,y,0)=\left\{\begin{array}{ll} h_1(x)dz^2, & \mbox{if } y=0,\\ h_2(y)dz^2, & \mbox{if }x=0.\end{array}\right.$$

\paragraph{The standard connections on the factors} The two factors $\pi_1:V_1\to X_1$ and $\pi_2:V_2\to X_2$ are both diffeomorphic to the standard trivial bundle $\matR^2\to\matR$,
and so can be seen as the usual tangent bundles $TX_1\cong T\matR\cong TX_2$. Thus, $g_1$ and $g_2$ are Riemannian metrics on them, and we can consider the usual
Levi-Civita connections $\nabla^1$ and $\nabla^2$ on them. Their Christoffel symbols are $\Gamma_{11}^1(g_1)=\frac{h_1'(x)}{2h_1(x)}$ and $\Gamma_{11}^1(g_2)=\frac{h_2'(y)}{2h_2(y)}$.
The formulae for $\nabla^1$ and $\nabla^2$ therefore are
$$\nabla^1(x,0,1)=\frac{h_1'(x)}{2h_1(x)}dx\otimes(x,0,1)\,\,\,\mbox{ and }\,\,\,\nabla^2(0,y,1)=\frac{h_2'(y)}{2h_2(y)}dy\otimes(0,y,1) $$
and in full form
$$\nabla^1(x,0,s_1(x))=\frac{h_1'(x)(s_1'(x)+s_1(x))}{2h_1(x)}dx\otimes(x,0,1),\,\,\,\,\nabla^2(0,y,s_2(y))=\frac{h_2'(y)(s_2'(y)+s_2(y))}{2h_2(y)}dy\otimes(0,y,1).$$

\paragraph{The resulting connection} In this specific case it is actually quite straightforward to assemble a connection on $V$ out of $\nabla^1$ and $\nabla^2$. The explicit formula
is as follows:
$$\nabla(x,y,s(x,y))=\left\{\begin{array}{ll}
\frac{h_1'(x)(\frac{\partial s}{\partial x}(x,0)+s(x,0))}{2h_1(x)}dx\otimes(x,0,1) & \mbox{if }y=0, \\
\frac{h_2'(y)(\frac{\partial s}{\partial y}(0,y)+s(0,y))}{2h_2(y)}dy\otimes(0,y,1) & \mbox{if }x=0, \\
\left(\frac{h_1'(0)(\frac{\partial s}{\partial x}(0,0)+s(0,0))}{2h_1(0)}dx+\frac{h_2'(0)(\frac{\partial s}{\partial y}(0,0)+s(0,0))}{2h_2(0)}dy\right)\otimes(0,0,1) & \mbox{if }x=y=0.
\end{array}\right.$$ Here $s$ can be, in particular, any smooth two-variable function; however, more generally it is a formal pair of functions $s_1$ (in variable $x$) and $s_2$ (in variable $y$)
such that $s_1(0)=s_2(0)$.

\paragraph{Observations on the example} The example just made give a rough idea of how one can obtain a connection on $V_1\cup_{\tilde{f}}V_2$ out of two given connections
on $V_1$ and $V_2$. On the other hand, it does not give a complete picture; indeed, the simplicity of the domain of gluing on the base spaces ensures that we do not have to impose
any conditions on $\nabla^1$ and $\nabla^2$, although later on we will see that a certain compatibility condition is needed.

\subsection{Covariant derivatives}

The usual notion of the covariant derivative of a section $s\in C^{\infty}(M,E)$ along a smooth vector field $X\in C^{\infty}(M,TM)$ extends easily to smooth sections $s\in C^{\infty}(X,V)$
of a diffeological vector pseudo-bundle. It suffices to specify that such derivatives are  with respect to smooth sections of the pseudo-bundle $(\Lambda^1(X))^*$.

\begin{defn}
Let $\pi:V\to X$ be a finite-dimensional diffeological vector pseudo-bundle, let $\nabla:C^{\infty}(X,V)\to C^{\infty}(X,\Lambda^1(X)\otimes V)$ be a diffeological connection on it, and let 
$t\in C^{\infty}(X,(\Lambda^1(X))^*)$ be a smooth section of the dual pseudo-bundle $(\Lambda^1(X))^*$. Let $s\in C^{\infty}(X,V)$; the \textbf{covariant derivative} of $s$ along $t$ is the section 
$\nabla s(t)=\nabla_t s$.
\end{defn}

\begin{lemma}
For any $t\in C^{\infty}(X,(\Lambda^1(X))^*)$ and for any $s\in C^{\infty}(X,V)$ we have $\nabla_t  s\in C^{\infty}(X,V)$.
\end{lemma}

\begin{proof}
This is obvious, since the diffeology on $(\Lambda^1(X))^*$, as on any dual pseudo-bundle, is defined so that the  evaluation functions $x\mapsto t(x)(\alpha^i(x))$ be smooth.
\end{proof}

We thus conclude that if $\nabla$ and $t$ are as above, $\nabla_t$ is well-defined as an operator $C^{\infty}(X,V)\to C^{\infty}(X,V)$. We furthermore have the following.

\begin{thm}
For any $t\in C^{\infty}(X,(\Lambda^1(X))^*)$ the map $\nabla_t:C^{\infty}(X,V)\to C^{\infty}(X,V)$ given by $s\mapsto\nabla_t s$ is smooth for the functional diffeology on $C^{\infty}(X,V)$.
\end{thm}

\begin{proof}
Let $p:U\to C^{\infty}(X,V)$ be a plot of $C^{\infty}(X,V)$. By the properties of a functional diffeology, the map $U\times X\to V$ given by $(u,x)\mapsto p(u)(x)$ is smooth, which also implies that for any plot 
$q:U'\to X$ the map $U\times U'\to V$ acting by $(u,u')\mapsto p(u)(q(u'))$ is a plot of $V$.

In order to prove that $\nabla_t$ is smooth, we need to show that $u\mapsto\nabla_t p(u)$ is a plot of $C^{\infty}(X,V)$. Since $\nabla$ is smooth as a map 
$C^{\infty}(X,V)\to C^{\infty}(X,\Lambda^1(X)\otimes V)$, its composition with any given plot $p$ of $C^{\infty}(X,V)$ is a plot of $C^{\infty}(X,\Lambda^1(X)\otimes V)$. This composition has form 
$u\mapsto\nabla p(u)$. It remains to notice that $u\mapsto \nabla_t(p(u))$ is the evaluation of it on the constant plot of $(\Lambda^1(X))^*$ with value $t$, which implies that it is a plot of $C^{\infty}(X,V)$, 
as wanted.
\end{proof}

There are also the expected linearity properties, stated below.

\begin{thm}
The operator $t\mapsto\nabla_t$ is $C^{\infty}(X,\matR)$-linear, that is, $\nabla_{t_1+t_2}=\nabla_{t_1}+\nabla_{t_2}$ and $\nabla_{f\cdot t}=f\nabla_t$ for any smooth function $f:X\to\matR$.
\end{thm}

\begin{proof}
This is a direct consequence of the definitions.
\end{proof}

\subsection{Compatibility with a pseudo-metric}

The usual notion of compatibility of a connection with a Riemannian metric extends trivially to the diffeological context. Let $\pi:V\to X$ be a diffeological vector pseudo-bundle that admits a pseudo-metric; 
let $g$ be a choice of a pseudo-metric on $V$, and let $\nabla$ be a connection on $V$.

\begin{defn}
The connection $\nabla$ is said to be \textbf{compatible with the pseudo-metric $g$} if for every two smooth sections $s,t$ of $\pi:V\to X$ we have that
$$d(g(s,t))=g(\nabla s,t)+g(s,\nabla t),$$ where for every $1$-form $\omega\in\Lambda^1(X)$ we set by definition 
$g(\omega\otimes s,t)=g(s,\omega\otimes t)=\omega\cdot g(s,t)$.
\end{defn}

The differential $d(g(s,t))$ in the above definition is meant as a section of $\Lambda^1(X)$ (as opposed to a form in $\Omega^1(X)$).

\subsection{Pseudo-bundle operations and diffeological connections}

The usual connections are well-behaved with respect to the standard operations, such those of direct sum, tensor product, or taking dual, on smooth vector bundles. In this section we show that the same 
is true of diffeological connections in the case of direct sums and tensor products, while the situation is more complicated for dual pseudo-bundles.

\subsubsection{Direct sum}

Let $\pi_1:V_1\to X$ and $\pi_2:V_2\to X$ be two diffeological vector pseudo-bundles over the same base space $X$. Suppose that each of them can be endowed with a connection; let
$$\nabla^1:C^{\infty}(X,V_1)\to C^{\infty}(X,\Lambda^1(X)\otimes V_1)\,\,\,\mbox{ and }\,\,\,\nabla^2:C^{\infty}(X,V_2)\to C^{\infty}(X,\Lambda^1(X)\otimes V_2).$$ Consider the direct sum pseudo-bundle
$\pi_1\oplus\pi_2:V_1\oplus V_2\to X$; let 
$$\mbox{pr}_{V_1}:V_1\oplus V_2\to V_1\,\,\,\mbox{ and }\,\,\,\mbox{pr}_{V_2}:V_1\oplus V_2\to V_2$$ be the standard direct sum projections, and let
$$\mbox{Incl}_{V_1}:V_1\cong V_1\oplus\{0\}\hookrightarrow V_1\oplus V_2\,\,\,\mbox{ and }\,\,\,\mbox{Incl}_{V_2}:V_2\cong\{0\}\oplus V_2\hookrightarrow V_1\oplus V_2$$ be the obvious inclusions. 
These maps are smooth by the definition of the diffeology on a direct sum of pseudo-bundles (see \cite{vincent}).

\begin{defn}
The \textbf{direct sum of the connections $\nabla^1$ and $\nabla^2$} is the operator 
$$\nabla^1\oplus\nabla^2:C^{\infty}(X,V_1\oplus V_2)\to C^{\infty}(X,\Lambda^1(X)\otimes(V_1\oplus V_2))$$ defined as follows. Let $s\in C^{\infty}(X,V_1\oplus V_2)$ be an arbitrary section; denote by 
$s_1:=\mbox{pr}_{V_1}\circ s$ and $s_2:=\mbox{pr}_{V_2}\circ s$. We define
\emph{$$(\nabla^1\oplus\nabla^2)s=(\mbox{Id}_{\Lambda^1(X)}\otimes\mbox{Incl}_{V_1})\circ(\nabla^1s_1)+(\mbox{Id}_{\Lambda^1(X)}\otimes\mbox{Incl}_{V_2})\circ(\nabla^2s_2).$$}
\end{defn} 

The final sum is of course taken in $\Lambda^1(X)\otimes(V_1\oplus V_2)$. The following then is an easy analogue of the standard fact.

\begin{prop}\label{direct:sum:of:connections:is:connection:prop}
Let $X$ be a diffeological space, let $\pi_1:V_1\to X$ and $\pi_2:V_2\to X$ be two diffeological vector pseudo-bundles over it, and let $\nabla^1$ and $\nabla^2$ be connections on $V_1$ and $V_2$ 
respectively. Then $\nabla^1\oplus\nabla^2$ is well-defined and is a connection on $V_1\oplus V_2$.
\end{prop}

\begin{proof}
The linearity property and the Leibnitz rule are established exactly as in the standard case, so we do not spell that out. What we really need to prove is that $\nabla^1\oplus\nabla^2$ is well-defined as 
a map into $C^{\infty}(X,\Lambda^1(X)\otimes(V_1\oplus V_2))$, and that it is smooth as a map $C^{\infty}(X,V_1\oplus V_2)\to C^{\infty}(X,\Lambda^1(X)\otimes(V_1\oplus V_2))$ for the respective 
functional diffeologies. Now, the former amounts to showing that for every section $s\in C^{\infty}(X,V_1\oplus V_2)$ we have that $(\nabla^1\oplus\nabla^2)s$ is a smooth map 
$X\to\Lambda^1(X)\otimes(V_1\oplus V_2)$. 

Let $p:U\to X$ be a plot of $X$; then $s\circ p$ is a plot of $V_1\oplus V_2$. We can assume that $U$ is small enough so that $s\circ p=q_1\oplus q_2$ for $q_i:U\to V_i$ a plot of $V_i$. Observe also that 
$s_i=\mbox{pr}_{V_i}\circ s\in C^{\infty}(X,V_i)$ and that by construction $s\circ p=q_1\oplus q_2=(s_1\circ p)\oplus (s_2\circ p)$. This allows us to obtain the following form for 
$(\nabla^1\oplus\nabla^2)s\circ p$: 
$$(\nabla^1\oplus\nabla^2)s\circ p=(\mbox{Id}_{\Lambda^1(X)}\otimes\mbox{Incl}_{V_1})\circ(\nabla^1s_1)p+(\mbox{Id}_{\Lambda^1(X)}\otimes\mbox{Incl}_{V_2})\circ(\nabla^2s_2)p.$$ Now, 
$(\nabla^1s_1)p$ is a plot of $\Lambda^1(X)\otimes V_1$ and $(\nabla^2s_2)p$ is a plot of $\Lambda^1(X)\otimes V_2$, hence $(\nabla^1\oplus\nabla^2)s\circ p$ is a plot of 
$\Lambda^1(X)\otimes(V_1\oplus V_2)$, and since $p$ is (a restriction of) any plot, this means that $(\nabla^1\oplus\nabla^2)s$ is smooth.

We now need to show that $\nabla^1\oplus\nabla^2:C^{\infty}(X,V_1\oplus V_2)\to C^{\infty}(X,\Lambda^1(X)\otimes(V_1\oplus V_2))$ is smooth. Let $q:U\to C^{\infty}(X,V_1\oplus V_2)$ be a plot for 
the functional diffeology on $C^{\infty}(X,V_1\oplus V_2)$. Observe that each $u\mapsto\mbox{pr}_{V_i}\circ q(u)$ is a plot of $C^{\infty}(X,V_i)$ for $i=1,2$; write $q_i$ for $\mbox{pr}_{V_i}\circ q$. We 
then have by construction that 
$$(\nabla^1\oplus\nabla^2)q(u)=(\mbox{Id}_{\Lambda^1(X)}\otimes\mbox{Incl}_{V_1})\circ(\nabla^1q_1(u))+(\mbox{Id}_{\Lambda^1(X)}\otimes\mbox{Incl}_{V_2})\circ(\nabla^2q_2(u)).$$
It remains to notice that each $u\mapsto\nabla^i q_i(u)$ is by assumption plots of $C^{\infty}(X,\Lambda^1(X)\otimes V_i)$, and that the post-composition with the fixed map 
$\mbox{Id}_{\Lambda^1(X)}\otimes\mbox{Incl}_{V_i}$ induces a smooth map $C^{\infty}(X,\Lambda^1(X)\otimes V_i)\to C^{\infty}(X,\Lambda^1(X)\otimes(V_1\oplus V_2))$, to conclude that 
$u\mapsto(\nabla^1\oplus\nabla^2)q(u)$ is a plot of the latter, which yields the final claim. 
\end{proof}

\subsubsection{Tensor product}

The case of the tensor product is analogous. Let again $X$ be a diffeological space, and let $\pi_1:V_1\to X$ and $\pi_2:V_2\to X$ be two diffeological vector pseudo-bundles over it. Consider the 
corresponding tensor product pseudo-bundle $\pi_1\otimes\pi_2:V_1\otimes V_2\to X$. Let also $\nabla^1$ and $\nabla^2$ be connections on $V_1$ and $V_2$ respectively. 

\begin{defn}
The \textbf{tensor product of the connections $\nabla^1$ and $\nabla^2$} is the operator 
$$\nabla^{\otimes}:C^{\infty}(X,V_1\otimes V_2)\to C^{\infty}(X,\Lambda^1(X)\otimes V_1\otimes V_2)$$ given by 
$$\nabla^{\otimes}:=\nabla^1\otimes\mbox{Id}_{C^{\infty}(X,V_2)}+\mbox{Id}_{C^{\infty}(X,V_1)}\otimes\nabla^2.$$
\end{defn}

The following then is a complete analogue of both the standard statement and of Proposition \ref{direct:sum:of:connections:is:connection:prop}, so we omit the proof.

\begin{prop}\label{tensor:product:of:connections:is:connection:prop}
Let $X$ be a diffeological space, let $\pi_1:V_1\to X$ and $\pi_2:V_2\to X$ be two diffeological vector pseudo-bundles over it, and let $\nabla^1$ and $\nabla^2$ be connections on $V_1$ and $V_2$ 
respectively. Then $\nabla^{\otimes}$ is well-defined and is a connection on $V_1\otimes V_2$.
\end{prop}

\section{Diffeological gluing and connections}

Let $\pi_1:V_1\to X_1$ and $\pi_2:V_2\to X_2$ be two diffeological vector pseudo-bundles, and let $(\tilde{f},f)$ be a gluing between them. Suppose furthermore that each of them can be endowed 
with a diffeological connection; let $\nabla^1$ and $\nabla^2$ be connections on $V_1$ and $V_2$ respectively. In this section we consider how, under specific assumptions on these connections, 
we can obtain a connection on $V_1\cup_{\tilde{f}}V_2$; the necessary assumptions take form, once again, of an appropriate compatibility notion, for which a preliminary construction is needed.

\subsection{The pullback map $f_{\Lambda}^*$ between the sub-bundles of $\Lambda^1(X_2)$ and $\Lambda^1(X_1)$}

Any connection on $V_1\cup_{\tilde{f}}V_2$ has the form of an operator 
$C^{\infty}(X_1\cup_f X_2,V_1\cup_{\tilde{f}}V_2)\to C^{\infty}(X_1\cup_f X_2,\Lambda^1(X_1\cup_f X_2)\otimes(V_1\cup_{\tilde{f}}V_2))$. In order to describe such an operator in terms of two operators 
of form $C^{\infty}(X_1,V_1)\to C^{\infty}(X_1,\Lambda^1(X_1)\otimes V_1)$ and $C^{\infty}(X_2,V_2)\to C^{\infty}(X_2,\Lambda^1(X_2)\otimes V_2)$, we are going to need an appropriate notion of a 
pullback map between certain subsets of $\Lambda^1(X_2)$ and $\Lambda^1(X_1)$.

\subsubsection{$f_{\Lambda}^*:\Lambda^1(X_2)\to\Lambda^1(X_1)$ for a smooth map $f:X_1\to X_2$}

The case when $f$ is defined on the entire $X_1$ is the simpler one; we consider it first. Recall that we already have the notion of a pullback map $f^*:\Omega^1(X_2)\to\Omega^1(X_1)$. together with 
$f^{-1}$ it gives a pseudo-bundle map between the trivial bundles $X_2\times\Omega^1(X_2)$ and $X_1\times\Omega^1(X_1)$, which acts (in an obvious manner) by 
$$(f^{-1},f^*)(x_2,\omega_2)=(f^{-1}(x_1),f^*\omega_2).$$

\begin{prop}\label{pullback:map:between:lambda:relative:diffeo:prop}
Let $X_1$ and $X_2$ be two diffeological spaces, and let $f:X_1\to X_2$ be a diffeomorphism. Then $f^*:\Omega^1(X_2)\to\Omega^1(X_1)$ induces a well-defined pullback map 
$\Lambda^1(X_2)\to\Lambda^1(X_1)$.
\end{prop}

\begin{proof}
Let $x_2\in X_2$, and let $\omega_2\in\Omega^1(X_2)$ be a form vanishing at $x_2$. We wish to know whether $f^*\omega_2$ vanishes at $f^{-1}(x_2)$. Consider a plot $p$ of $X_1$ centered at this 
point; then trivially $f\circ p$ is a plot of $X_2$ centered at $x_2$. Furthermore, we have
$$f^*\omega_2(p)(f^{-1}(x_2))=\omega_2(f\circ p)(x_2)=0.$$ Since $p$ and $x_2$ are arbitrary, we conclude that $f^*(\Omega_{x_2}^1(X_2))\subseteq\Omega_{f^{-1}(x_2)}^1(X_1)$. In fact, since $f$ 
is a diffeomorphism, we can apply the analogous reasoning to $f^{-1}$, obtaining $$f^*(\Omega_{x_2}^1(X_2))=\Omega_{f^{-1}(x_2)}^1(X_1).$$

It follows from what has been established in the previous paragraph it is obvious that $f^*$ yields a pseudo-bundle map between the two sub-bundles (of $X_2\times\Omega^1(X_2)$ and 
$X_1\times\Omega^1(X_1)$ respectively) consisting of vanishing forms:
$$X_2\times\Omega^1(X_2)\supseteq\left(\bigcup_{x_2\in X_2}\{x_2\}\times\Omega_{x_2}^1(X_2)\right)\to
\left(\bigcup_{x_1\in X_1}\{x_1\}\times\Omega_{x_1}^1(X_1)\right)\subseteq X_1\times\Omega^1(X_1);$$ this pseudo-bundle map covers $f^{-1}$. Therefore $f^*$ descends to a well-defined map on the 
quotient pseudo-bundles
$$(X_2\times\Omega^1(X_2))/\left(\bigcup_{x_2\in X_2}\{x_2\}\times\Omega_{x_2}^1(X_2)\right)\to(X_1\times\Omega^1(X_1))/\left(\bigcup_{x_1\in X_1}\{x_1\}\times\Omega_{x_1}^1(X_1)\right).$$
It remains to recall our prior observation that these quotients are precisely the corresponding $\Lambda^1$-bundles, that is,
$$(X_2\times\Omega^1(X_2))/\left(\bigcup_{x_2\in X_2}\{x_2\}\times\Omega_{x_2}^1(X_2)\right)\cong\Lambda^1(X_2),\,\,\,\,
(X_1\times\Omega^1(X_1))/\left(\bigcup_{x_1\in X_1}\{x_1\}\times\Omega_{x_1}^1(X_1)\right)\cong\Lambda^1(X_1),$$ whence the claim.
\end{proof} 

The final conclusion is that the pullback map is well-defined as a map $f_{\Lambda}^*:\Lambda^1(X_2)\to\Lambda^1(X_1)$; we do not introduce a separate notation for it, since it will always be clear from the 
context whether we mean the pullback map defined between the $\Omega^1(X_i)$'s or the $\Lambda^1(X_i)$'s.

\subsubsection{The map $f_{\Lambda}^*$ for $f:X_1\supset Y\to f(Y)\subseteq X_2$}

Let us now consider the general case. Suppose that $f$ is defined on a proper subset of $X_1$, and that its image is an \emph{a priori} proper subset of $X_2$. There is of course again a well-defined 
pullback map $f^*$ but it is not defined on the whole $\Lambda^1(X_2)$. We shall relate the domain and the range of $f^*$ to certain subsets of $\Lambda^1(X_2)$ and $\Lambda^1(X_1)$, and use it 
for an alternative description of the compatibility of elements of $\Lambda^1(X_i)$, in a form suitable for defining subsequently the compatibility of connections on pseudo-bundles over $X_1$ and $X_2$.

\paragraph{The properties of the pullback map $f_{\Lambda}^*:\Lambda^1(f(Y))\to\Lambda^1(Y)$} Considering $Y$ and $f(Y)$ as diffeological spaces for their natural subset diffeologies, it follows from 
Proposition \ref{pullback:map:between:lambda:relative:diffeo:prop} that there is the pullback map 
$$f_{\Lambda}^*:\Lambda^1(f(Y))\to\Lambda^1(Y);$$ its precursor is the pullback map $f^*:\Omega^1(f(Y))\to\Omega^1(Y)$. There is a natural commutativity between these two versions of $f^*$ expressed by
$$\pi_Y^{\Omega,\Lambda}\circ(f^{-1},f^*)=f_{\Lambda}^*\circ\pi_{f(Y)}^{\Omega,\Lambda},$$ where 
\begin{itemize} 
\item $\pi_Y^{\Omega,\Lambda}:Y\times\Omega^1(Y)\to\Lambda^1(Y)$ is the defining projection of $\Lambda^1(Y)$, and 
\item $\pi_{f(Y)}^{\Omega,\Lambda}:f(Y)\times\Omega^1(f(Y))\to\Lambda^1(f(Y))$ is the defining projection of $\Lambda^1(f(Y))$. 
\end{itemize}
The two compositions are defined on $f(Y)\times\Omega^1(f(Y))$.

\paragraph{The pullback map $f_{\Lambda}^*$ and the compatibility of elements of $\Lambda^1(X_1)$ and $\Lambda^1(X_2)$} Consider now the natural inclusions $i:Y\hookrightarrow X_1$ and 
$j:f(Y)\hookrightarrow X_2$; these give rise to the pullback maps $i^*:\Omega^1(X_1)\to\Omega^1(Y)$ and $j^*:\Omega^1(X_2)\to\Omega^1(f(Y))$ (note that in general they may not be surjective). 

\begin{lemma}\label{dual:inclusion:on:lambda:left:lem}
The map $(i^{-1},i^*):i(Y)\times\Omega^1(X_1)\to Y\times\Omega^1(Y)$ descends to a well-defined map $i_{\Lambda}^*:\Lambda^1(X_1)\supset(\pi_1^{\Lambda})^{-1}(Y)\to\Lambda^1(Y)$; in particular,
$$\pi_Y^{\Omega,\Lambda}\circ(i^{-1},i^*)=i_{\Lambda}^*\circ\pi_1^{\Omega,\Lambda}|_{i(Y)\times\Omega^1(X_1)}.$$
\end{lemma}

\begin{proof}
It suffices to show that $i^*$ preserves the vanishing of $1$-forms. Let $y\in Y$ and let $\omega_1\in\Omega_y^1(X_1)$. We need to show that $i^*\omega_1$ vanishes at $y$, so let $p:U\to Y$ be a plot 
centered at $y$, $p(0)=y$. Let us calculate $(i^*\omega_1)(p)(0)=\omega_1(i\circ p)(0)=0$, because by assumption $\omega_1$ vanishes at $y$/$i(y)$ and $i\circ p$ is obviously a plot of $X_1$ centered 
at $y$. Thus, $i^*(\Omega_y^1(X_1))\subseteq\Omega_y^1(Y)$, whence the claim.    
\end{proof}

A completely analogous statement is also true for the other factor.

\begin{lemma}\label{dual:inclusion:on:lambda:right:lem}
The map $(j^{-1},j^*):j(f(Y))\times\Omega^1(X_2)\to f(Y)\times\Omega^1(f(Y))$ descends to a well-defined map $j_{\Lambda}^*:\Lambda^1(X_2)\supset(\pi_2^{\Lambda})^{-1}(f(Y))\to\Lambda^1(f(Y))$ 
such that
$$\pi_{f(Y)}^{\Omega,\Lambda}\circ(j^{-1},j^*)=j_{\Lambda}^*\circ\pi_2^{\Omega,\Lambda}|_{j(f(Y))\times\Omega^1(X_2)}.$$
\end{lemma}

Recall now (\cite{forms-gluing}) that $\omega_1$ and $\omega_2$ are compatible if and only if $$i^*\omega_1=f^*(j^*\omega_2).$$ Let $y\in Y$ be arbitrary, and let
$\alpha_1=\omega_1+\Omega_y^1(X_1)\in\Lambda_y^1(X_1)$ and $\alpha_2=\omega_2+\Omega_{f(y)}^1(X_2)\in\Lambda_{f(y)}^1(X_2)$ be two compatible elements of $\Lambda^1(X_1)$ and
$\Lambda^1(X_2)$. The compatibility condition for such elements means that any pair $(\omega_1',\omega_2')$, where $\omega_1'\in\alpha_1$ and $\omega_2'\in\alpha_2$, is a compatible one, that is, 
by the aforementioned criterion 
$$i^*(\omega_1')=f^*(j^*(\omega_2'))\,\,\mbox{ for all }\omega_1'\in\alpha_1\,\mbox{ and }\,\omega_2'\in\alpha_2.$$

\begin{prop}\label{criterion:of:compatibility:of:elements:of:lambda:prop}
Two elements $\alpha_1\in\Lambda_y^1(X_1)$ and $\alpha_2\in\Lambda_{f(y)}^1(X_2)$ are compatible if and only if the following is true: $$i_{\Lambda}^*\alpha_1=f_{\Lambda}^*(j_{\Lambda}^*\alpha_2).$$
\end{prop} 

\begin{proof}
This follows from Lemmas \ref{dual:inclusion:on:lambda:left:lem} and \ref{dual:inclusion:on:lambda:left:lem}. Indeed, by construction there exist $\omega_1\in\Omega^1(X_1)$ and 
$\omega_2\in\Omega^1(X_2)$ such that $\alpha_1=\pi_1^{\Omega,\Lambda}(y,\omega_1)$ and $\alpha_2=\pi_2^{\Omega,\Lambda}(f(y),\omega_2)$; furthermore, $\alpha_1$ and $\alpha_2$ are 
compatible if and only if any two such $\omega_1$ and $\omega_2$ are compatible. By Lemma \ref{dual:inclusion:on:lambda:left:lem}, Lemma \ref{dual:inclusion:on:lambda:left:lem}, and the construction 
of the pullback map $f^*:\Omega^1(f(Y))\to\Omega^1(Y)$ we then have 
$$i_{\Lambda}^*\alpha_1=\pi_Y^{\Omega,\Lambda}(y,i^*\omega_1)\,\,\,\mbox{ and }\,\,\,
f_{\Lambda}^*(j_{\Lambda}^*\alpha_2)=f^*(\pi_{f(Y)}^{\Omega,\Lambda}(f(y),j^*\omega_2))=\pi_Y^{\Omega,\Lambda}(y,f^*(j^*\omega_2)).$$ These expressions are equal for all choices of 
$\omega_i\in\alpha_i$ if and only if $\alpha_1^*$ and $\alpha_2^*$ are compatible, by the definition of compatibility of elements of $\Lambda^1(X_1)$ and $\Lambda^1(X_2)$, and the aforementioned 
criterion of compatibility of elements of $\Omega^1(X_1)$ and $\Omega^1(X_2)$; this completes the proof.
\end{proof}

Proposition \ref{criterion:of:compatibility:of:elements:of:lambda:prop} provides our main criterion for compatibility of elements in $\Lambda^1(X_1)$ and $\Lambda^1(X_2)$, in the form suitable for defining 
compatible connections (which is one of our main goals); we do this in the section immediately below.

\subsection{The induced connection on $V_1\cup_{\tilde{f}}V_2$}

Let $\pi_1:V_1\to X_1$ and $\pi_2:V_2\to X_2$ be two diffeological vector pseudo-bundles, and let $(\tilde{f},f)$ be a gluing between them such that both $\tilde{f}$ and $f$ are diffeomorphisms of their 
domains with their images. Given a connection $\nabla^1$ on $V_1$ and a connection $\nabla^2$ on $V_2$, we might be able to obtain out of them an induced connection on $V_1\cup_{\tilde{f}}V_2$; 
for this to be feasible, the two connections must be subject to some restrictions, which are expressed by the appropriate compatibility notion. After describing this notion, we provide the construction of 
the induced connection, proving that it is indeed a connection.

\subsubsection{The definition of compatible connections}

The idea behind the compatibility notion for connections $\nabla^1$ and $\nabla^2$ on $V_1$ and $V_2$ is as follows. Let $s_1\in C^{\infty}(X_1,V_1)$ and $s_2\in C^{\infty}(X_2,V_2)$; let $y\in Y$. Then 
$(\nabla^1s_1)(y)=\sum\alpha^i\otimes v_i$ for some $\alpha^i\in\Lambda_y^1(X_1)$ and $v_i\in V_1$; likewise, $(\nabla^2s_2)(f(y))=\sum_j\beta^j\otimes w_j$ for $\beta_j\in\Lambda_{f(y)}^1(X_2)$ and 
$w_j\in V_2$. Now, $(\nabla^1s_1)(y)$ and $(\nabla^2s_2)(f(y))$ can be easily identified with certain elements of 
$$\left(\Lambda_y^1(X_1)\oplus\Lambda_{f(y)}^1(X_2)\right)\otimes\left(V_1\cup_{\tilde{f}}V_2\right)_{i_2(f(y))};$$ this direct sum contains the corresponding fibre of 
$\Lambda^1(X_1\cup_f X_2)\otimes(V_1\cup_{\tilde{f}}V_2)$ as a (generally proper) subspace.

\begin{defn}\label{compatible:connections:defn}
Let $\pi_1:V_1\to X_1$ and $\pi_2:V_2\to X_2$ be two diffeological vector pseudo-bundles, let $f$ and $\tilde{f}$ be maps defining a gluing of the former to the latter, each of which is a diffeomorphism 
of its domain with its image, and let $Y$ be the domain of definition of $f$. Let $\nabla^1$ be a connection on $V_1$, and let $\nabla^2$ be a connection on $V_2$. We say that $\nabla^1$ and $\nabla^2$ 
are \textbf{compatible} if for any pair $s_1\in C^{\infty}(X_1,V_1)$ and $s_2\in C^{\infty}(X_2,V_2)$ of compatible sections, and for any $y\in Y$, we have
$$\left((i_{\Lambda}^*\otimes\tilde{f})\circ(\nabla^1s_1)\right)(y)=\left(((f_{\Lambda}^*j_{\Lambda}^*)\otimes\mbox{Id}_{V_2})\circ(\nabla^2s_2)\right)(f(y)).$$
\end{defn}

We can now better formulate our reason for defining the compatibility of connections in the way we just, by stating the following.

\begin{prop}\label{sum:of:images:compatible:sections:is:in:glued:prop}
Let $\nabla^1$ and $\nabla^2$ be compatible connections on $V_1$ and $V_2$ respectively. Then for any compatible sections $s_1\in C^{\infty}(X_1,V_1)$ and $s_2\in C^{\infty}(X_2,V_2)$
and for any $y\in Y$ we have
$$\left((\mbox{Id}_{\Lambda_y^1(X_1)}\otimes\tilde{f})\circ(\nabla^1s_1)\right)(y)\oplus(\nabla^2s_2)(f(y))\in\left(\Lambda_y^1(X_1)\oplus_{comp}\Lambda_{f(y)}^1(X_2)\right)\otimes V_2.$$
\end{prop}

\begin{proof}
The statement of the proposition expresses the fact that two elements $\alpha_1\in\Lambda_y^1(X_1)$ and $\alpha_2\in\Lambda_{f(y)}^1(X_2)$ are compatible if and only if
$i_{\Lambda}^*\alpha_1=f_{\Lambda}^*(j_{\Lambda}^*\alpha_2)$, and this is the content of Proposition \ref{criterion:of:compatibility:of:elements:of:lambda:prop}.
\end{proof}

\subsubsection{The induced connection $\nabla^{\cup}$}

Let the two pseudo-bundles $V_1$ and $V_2$ be endowed with connections $\nabla^1$ and $\nabla^2$, and assume that these connections are compatible in the sense of Definition 
\ref{compatible:connections:defn}. We shall first describe the connection on $V_1\cup_{\tilde{f}}V_2$ induced by them and then prove that it is, indeed, a connection. Recall that all throughout we assume 
that all bluings are along diffeomorphisms.

\paragraph{The definition of $\nabla^{\cup}$} Let $x\in i_2(f(Y))$, and recall that the map $\tilde{i}_1:X_1\to X_1\cup_f X_2$ defined as the composition of the natural inclusion 
$X_1\hookrightarrow X_1\sqcup X_2$ with the defining quotient projection $X_1\sqcup X_2\to X_1\cup_f X_2$ is an inclusion if $f$ is a diffeomorphism. The connection $\nabla^{\cup}$ is then defined as 
follows.

\begin{defn}\label{induced:connection:on:result:of:gluing:defn}
Let $\pi_1:V_1\to X_1$ and $\pi_2:V_2\to X_2$ be two diffeological vector pseudo-bundles, let $(\tilde{f},f)$ be a gluing between them such that $\tilde{f}$ is a diffeomorphism and $f$, also a diffeomorphism, 
is such that $\calD_1^{\Omega}=\calD_2^{\Omega}$, and let $\nabla^1$ and $\nabla^2$ be compatible connections on $V_1$ and $V_2$ respectively. The \textbf{induced connection $\nabla^{\cup}$} on 
$V_1\cup_{\tilde{f}}V_2$ is the operator defined as follows. Let $s\in C^{\infty}(X_1\cup_f X_2,V_1\cup_{\tilde{f}}V_2)$ be a section. Since $f$ and $\tilde{f}$ are diffeomorphisms, it has a unique presentation 
of form $s=s_1\cup_{(f,\tilde{f})}s_2$ for $s_1\in C^{\infty}(X_1,V_1)$ and $s_2\in C^{\infty}(X_2,V_2)$. Then
\begin{flushleft}
$(\nabla^{\cup}s)(x)=$
\end{flushleft}
$$\left\{\begin{array}{cl}
\left((\tilde{\rho}_1^{\Lambda})^{-1}\otimes j_1\right)\left((\nabla^1s_1)(i_1^{-1}(x))\right) & \mbox{for }x\in i_1(X_1\setminus Y), \\
((\tilde{\rho}_1^{\Lambda}\oplus\tilde{\rho}_2^{\Lambda})^{-1}\otimes\mbox{Id}_{V_1\cup_{\tilde{f}}V_2})
(\left(\mbox{Id}_{\Lambda^1(X_1)}\otimes(j_2\circ\tilde{f})\right)\left((\nabla^1s_1)(\tilde{i}_1^{-1}(x))\right)\oplus & \\
\oplus\left(\mbox{Id}_{\Lambda^1(X_2)}\otimes j_2\right)\left((\nabla^2s_2)(i_2^{-1}(x))\right)) & 
\mbox{for }x\in i_2(f(Y)), \\ 
\left((\tilde{\rho}_2^{\Lambda})^{-1}\otimes j_2\right)\left((\nabla^2s_2)(i_2^{-1}(x))\right) & \mbox{for }x\in i_2(X_2\setminus f(Y)).
\end{array}\right.$$
\end{defn}

\paragraph{Proof that $\nabla^{\cup}$ is a connection} Two items need to be checked: one, that $\nabla^{\cup}$ is well-defined as a map
$$C^{\infty}(X_1\cup_f X_2,V_1\cup_{\tilde{f}}V_2)\to C^{\infty}(X_1\cup_f X_2,\Lambda^1(X_1\cup_f X_2)\otimes(V_1\cup_{\tilde{f}}V_2)),$$ and two, that it is smooth for the functional
diffeologies on these two spaces.

\begin{lemma}\label{nabla-s:is:map:into:lambda-glued:lem}
For every section $s\in C^{\infty}(X_1\cup_f X_2,V_1\cup_{\tilde{f}}V_2)$ and for every $x\in X_1\cup_f X_2$ we have $(\nabla^{\cup}s)(x)\in\Lambda^1(X_1\cup_f X_2)\otimes(V_1\cup_{\tilde{f}}V_2)$.
\end{lemma}

\begin{proof}
We shall consider separately the cases when $x\in i_1(X_1\setminus Y)$, $x\in i_2(X_2\setminus f(Y))$, and $x\in i_2(f(Y))$; the former two are actually analogous, so it suffices to treat just one of them. 
Let $x\in i_1(X_1\setminus Y)$. Then by construction 
$$(\nabla^{\cup}s)(x)=\left((\tilde{\rho}_1^{\Lambda})^{-1}\otimes j_1\right)\left((\nabla^1s_1)(i_1^{-1}(x))\right).$$ Since $\nabla^1$ is a connection on $V_1$, we have that 
$(\nabla^1s_1)(i_1^{-1}(x))\in\Lambda^1(X_1)\otimes V_1$. Its image under the map $(\tilde{\rho}_1^{\Lambda})^{-1}\otimes j_1$ belongs to $\Lambda^1(X_1\cup_f X_2)\otimes(V_1\cup_{\tilde{f}}V_2)$ 
by the definition of this map. As just mentioned, the case of $x\in i_2(X_2\setminus f(Y))$ is completely analogous.

Let $x\in i_2(f(Y))$. To abbreviate the lengthy expression for $(\nabla^{\cup}s)(x)$, let us write $y:=\tilde{i}_1^{-1}(x)$ and $y'=i_2^{-1}(x)$. Since the expression for $(\nabla^{\cup}s)(x)$ involves both 
$\nabla^1s_1$ and $\nabla^2s_2$, and $s_1$ and $s_2$ are compatible, we can draw the desired conclusion from Proposition \ref{sum:of:images:compatible:sections:is:in:glued:prop}.
\end{proof}

Thus, $\nabla^{\cup}s$ is always well-defined as a map $X_1\cup_f X_2\to\Lambda^1(X_1\cup_f X_2)\otimes(V_1\cup_{\tilde{f}}V_2)$. Next, we need to show that it is actually smooth. 

\begin{lemma}\label{nabla:on:section:of:glued:is:smooth:lem}
For every section $s\in C^{\infty}(X_1\cup_f X_2,V_1\cup_{\tilde{f}}V_2)$ the section $\nabla^{\cup}s:X_1\cup_f X_2\to\Lambda^1(X_1\cup_f X_2)\otimes(V_1\cup_{\tilde{f}}V_2)$ is smooth, that is, 
$\nabla^{\cup}s\in C^{\infty}(X_1\cup_f X_2,\Lambda^1(X_1\cup_f X_2)\otimes(V_1\cup_{\tilde{f}}V_2))$.
\end{lemma}

\begin{proof}
Showing that $\nabla^{\cup}s$ is smooth amounts to showing that for any arbitrary plot $p:U\to X_1\cup_f X_2$ of $X_1\cup_f X_2$ the composition $(\nabla^{\cup}s)\circ p$ is a plot of 
$\Lambda^1(X_1\cup_f X_2)\otimes(V_1\cup_{\tilde{f}}V_2)$. As usual, we can assume that $U$ is connected, so that $p$ lifts to either a plot $p_1$ of $X_1$ or to a plot $p_2$ of $X_2$; accordingly, 
for any $u\in U$ either
$$p(u)=\left\{\begin{array}{ll} i_1(p_1(u)) & \mbox{if }p_1(u)\in X_1\setminus Y, \\ i_2(f(p_1(u))) & \mbox{if }p_1(u)\in Y  \end{array}\right.\,\,\,\mbox{ or }\,\,\,p(u)=i_2(p_2(u)).$$ 

Assume first that $p$ lifts to $p_1$. Then
$$(\nabla^{\cup}s)(p(u))=\left\{\begin{array}{cl} 
\left((\tilde{\rho}_1^{\Lambda})^{-1}\otimes j_1\right)\left((\nabla^1s_1)(p_1(u))\right) & \mbox{if }p_1(u)\in X_1\setminus Y, \\
((\tilde{\rho}_1^{\Lambda}\oplus\tilde{\rho}_2^{\Lambda})^{-1}\otimes\mbox{Id}_{V_1\cup_{\tilde{f}V_2}})
(\left(\mbox{Id}_{\Lambda^1(X_1)}\otimes(j_2\circ\tilde{f})\right)\left((\nabla^1s_1)(p_1(u))\right)\oplus & \\
\oplus\left(\mbox{Id}_{\Lambda^1(X_2)}\otimes j_2\right)\left((\nabla^2s_2)(f(p_1(u)))\right)) & \mbox{if }p_1(u)\in Y.
\end{array}\right.$$ By Theorem 1.9 and the definition of the tensor product of diffeological vector pseudo-bundles, to check that this is a plot of $\Lambda^1(X_1\cup_f X_2)\otimes(V_1\cup_{\tilde{f}}V_2)$, 
it suffices to check that its composition with $\tilde{\rho}_1^{\Lambda}\otimes\mbox{Id}_{V_1\cup_{\tilde{f}}V_2}$ is a plot of $\Lambda^1(X_1)\otimes(V_1\cup_{\tilde{f}}V_2)$ and that the composition with
$\tilde{\rho}_2^{\Lambda}\otimes\mbox{Id}_{V_1\cup_{\tilde{f}}V_2}$, where defined, is smooth as a map into $\Lambda^1(X_2)\otimes(V_1\cup_{\tilde{f}}V_2)$, for the \emph{subset} diffeology on 
$p_1^{-1}(Y)\subseteq U$. 

The composition of $(\nabla^{\cup}s)\circ p$ with $\tilde{\rho}_1^{\Lambda}\otimes\mbox{Id}_{V_1\cup_{\tilde{f}}V_2}$ has form $(\mbox{Id}_{\Lambda^1(X_1)}\otimes j_1)\circ(\nabla^1s_1)\circ p_1$ at 
points of $p_1^{-1}(X_1\setminus Y)$. Since over $i_2(f(Y))$ the map $\tilde{\rho}_1^{\Lambda}$ acts by the projection of the direct sum $\Lambda_y^1(X_1)\oplus\Lambda_{f(y)}^1(X_2)$ onto its first 
factor, for $u\in p_1^{-1}(Y)$ this composition has form
\begin{flushleft}
$\left(\tilde{\rho}_1^{\Lambda}\otimes\mbox{Id}_{V_1\cup_{\tilde{f}}V_2}\right)\circ(\nabla^{\cup}s)\circ p=$
\end{flushleft}
\begin{flushright}
$=\left(\tilde{\rho}_1^{\Lambda}\otimes\mbox{Id}_{V_1\cup_{\tilde{f}}V_2}\right)\circ\left(\mbox{Incl}_{\Lambda_{f^{-1}(i_2^{-1}(x))}^1(X_1)}\otimes(j_2\circ\tilde{f})\right)\circ(\nabla^1s_1)\circ p_1
=(\mbox{Id}_{\Lambda^1(X_1)}\otimes(j_2\circ\tilde{f}))\circ(\nabla^1s_1)\circ p_1$.
\end{flushright} Thus, the complete form of the composition under consideration is
$$\left(\tilde{\rho}_1^{\Lambda}\otimes\mbox{Id}_{V_1\cup_{\tilde{f}}V_2}\right)\circ(\nabla^{\cup}s)\circ p=\left\{\begin{array}{ll} 
(\mbox{Id}_{\Lambda^1(X_1)}\otimes j_1)\circ(\nabla^1s_1)\circ p_1 & \mbox{for }u\mbox{ such that }p_1(u)\in X_1\setminus Y, \\
(\mbox{Id}_{\Lambda^1(X_1)}\otimes(j_2\circ\tilde{f}))\circ(\nabla^1s_1)\circ p_1 & \mbox{for }u\mbox{ such that }p_1(u)\in Y.
\end{array}\right.$$ Since $(\nabla^1s_1)\circ p_1$ is a plot of $\Lambda^1(X_1)\otimes V_1$ by assumption, it suffices to recall that $\left\{\begin{array}{l} j_1 \\ j_2\circ\tilde{f} \end{array}\right.$ is a 
smooth inclusion of $V_1$ into $V_1\cup_{\tilde{f}}V_2$.

Let us now consider the composition $(\tilde{\rho}_2^{\Lambda}\otimes\mbox{Id}_{V_1\cup_{\tilde{f}}V_2})\circ(\nabla^{\cup}s)\circ p$. This is defined only for $u$ such that $p_1(u)\in Y$; using the defintiion 
of $\tilde{\rho}_2^{\Lambda}$, the restriction of this composition to $p_1^{-1}(Y)\subseteq U$ has form
$$\left((\tilde{\rho}_2^{\Lambda}\otimes\mbox{Id}_{V_1\cup_{\tilde{f}}V_2})\circ(\nabla^{\cup}s)\circ p\right)|_{p_1^{-1}(Y)}=
(\tilde{\rho}_2^{\Lambda}\otimes\mbox{Id}_{V_1\cup_{\tilde{f}}V_2})\circ(\mbox{I}_{\Lambda^1(X_2)}\otimes j_2)\circ(\nabla^2s_2)\circ(f\circ p_1).$$ We need to show that this is a plot relative 
to the subset diffeology on $p_1^{-1}(Y)$, that is, if $p_1':U'\to U$ is a usual smooth map whose range is contained in $p_1^{-1}(Y)$ then 
$(\tilde{\rho}_2^{\Lambda}\otimes\mbox{Id}_{V_1\cup_{\tilde{f}}V_2})\circ(\nabla^{\cup}s)\circ (p\circ p_1')$ must be a plot of $\Lambda^1(X_2)\otimes(V_1\cup_{\tilde{f}}V_2)$. We then have 
$$(\tilde{\rho}_2^{\Lambda}\otimes\mbox{Id}_{V_1\cup_{\tilde{f}}V_2})\circ(\nabla^{\cup}s)\circ (p\circ p_1')=
(\tilde{\rho}_2^{\Lambda}\otimes\mbox{Id}_{V_1\cup_{\tilde{f}}V_2})\circ(\mbox{Id}_{\Lambda^1(X_2)}\otimes j_2)\circ(\nabla^2s_2)\circ(f\circ p_1\circ p_1'),$$ and it suffices to observe that 
$f\circ p_1\circ p_1'$ is a plot of $X_2$, since by assumption $f$ is smooth, $p_1\circ p_1'$ is a plot of $X_1$ by the axioms of diffeology, and its range is contained in $Y$ by construction. Thus, 
it follows from the assumption on $\nabla^2$ that $(\tilde{\rho}_2^{\Lambda}\otimes\mbox{Id}_{V_1\cup_{\tilde{f}}V_2})\circ(\nabla^{\cup}s)\circ (p\circ p_1')$ is indeed a plot of 
$\Lambda^1(X_2)\otimes(V_1\cup_{\tilde{f}}V_2)$, as wanted, which completes the treatment of the case when $p$ lifts to a plot of $X_1$.

If $p$ lifts to a plot $p_2$ of $X_2$, the proof is completely analogous, so we avoid spelling it out, ending the proof with this remark.
\end{proof}

We shall check next the standard linearity properties of $\nabla^{\cup}$.

\begin{lemma}\label{nabla:satisfies:leibnitz:lem}
The operator $\nabla^{\cup}$ is linear and satisfies the Leibnz rule.
\end{lemma}

\begin{proof}
All maps, as well as operations, involved in the construction of $\nabla^{\cup}$ are fibrewise additive, so the additivity of $\nabla^{\cup}$ is obvious. Let us check that $\nabla^{\cup}$ satisfies the Leibniz rule. 
Let $h\in C^{\infty}(X_1\cup_f X_2,\matR)$, and let $s\in C^{\infty}(X_1\cup_f X_2,V_1\cup_{\tilde{f}}V_2)$. Define $h_1\in C^{\infty}(X_1,\matR)$ and $h_2\in C^{\infty}(X_2,\matR)$ by
$$h_1(x_1)=\left\{\begin{array}{ll} h(i_1(x_1)) & \mbox{if }x_1\in X_1\setminus Y, \\ h(i_2(f(x_1))) & \mbox{if }x_1\in Y \end{array}\right.,\,\,\mbox{ and }\,\,h_2(x_2)=h(i_2(x_2))\mbox{ for all }x_2\in X_2.$$ 
Notice that this corresponds to the presentation of $h$ as $h=h_1\cup_f h_2$, already mentioned in Section 2.2.4. Recall also that by Theorem 2.29 $s$ admits a presentation as $s=s_1\cup_{(f,\tilde{f})}s_2$ 
for some $s_1\in C^{\infty}(X_1,V_1)$ and $s_2\in C^{\infty}(X_2,V_2)$, that in our present case (of gluing along two diffeomorphisms) are also uniquely defined. Finally, recall from Section 2.2.4 that 
$$hs=(h_1\cup_f h_2)(s_1\cup_{(f,\tilde{f})}s_2)=(h_1s_1)\cup_{(f,\tilde{f})}(h_2s_2).$$ 

By assumption $\nabla^1$ is a connection, so we have that $\nabla^1(h_1s_1)=dh_1\otimes s_1+h_1(\nabla^1s_1)$, and likewise, $\nabla^2$ being a connection as well, we have that 
$\nabla^2(h_2s_2)=dh_2\otimes s_2+h_2(\nabla^2s_2)$. Thus, it suffices to check that 
$$(\tilde{\rho}_1^{\Lambda})^{-1}(dh_1(x))=dh(i_1(x))\mbox{ for all }x\in X_1\setminus Y\,\,\,\mbox{ and }\,\,\,(\tilde{\rho}_2^{\Lambda})^{-1}(dh_2(x))=dh(i_2(x))$$ to obtain the desired equality 
$\nabla^{\cup}(hs)=dh\otimes s+h(\nabla^{\cup}s)$. Let us consider the first of these equalities, in its equivalent form $\tilde{\rho}_1^{\Lambda}(dh(i_1(x)))=dh_1(x)$.

Recall that, as a section of $\Lambda^1(X_1\cup_f X_2)$, the differential $dh$ is defined by $dh(x)=\pi^{\Omega,\Lambda}(x,dh)$ for all $x\in X_1\cup_f X_2$, where $dh$ on the right stands for the element of 
$\Omega^1(X_1\cup_f X_2)$ given by $dh(p)=d(h\circ p)$ for any plot $p$ of $X_1\cup_f X_2$. Thus, we can also write $dh(x)=dh+\Omega_x^1(X_1\cup_f X_2)$. Likewise, $dh_1$, as a section of 
$\Lambda^1(X_1)$, is given by $dh_1(x_1)=\pi_1^{\Omega,\Lambda}(x_1,dh_1)$, with, on the right, $dh_1\in\Omega^1(X_1)$ being given by $dh_1(p_1)=d(h_1\circ p_1)$ for any plot $p_1$ of $X_1$, and
equivalently, $dh_1:X_1\to\Lambda^1(X_1)$ is given by $dh_1(x_1)=dh_1+\Omega_{x_1}^1(X_1)$. 

Corresponding to the inclusion $\tilde{i}_1$ is the pullback map $\tilde{i}_1^*:\Omega^1(X_1\cup_f X_2)\to\Omega^1(X_1)$. It is easily seen that $\tilde{i}_1^*$ is a lift of $\tilde{\rho}_1^{\Lambda}$,
 \emph{i.e.}, that the following is true:
$$\tilde{\rho}_1^{\Lambda}\circ\pi^{\Omega,\Lambda}=\pi_1^{\Omega,\Lambda}\circ(\tilde{i}_1^{-1},\tilde{i}_1^*)$$ wherever this expression makes sense, that is, on the direct product 
$X_1\times\Omega^1(X_1\cup_f X_2)$. Putting everything together, we obtain
$$\tilde{\rho}_1^{\Lambda}(dh(\tilde{i}_1(x)))=\tilde{\rho}_1^{\Lambda}(\pi^{\Omega,\Lambda}(\tilde{i}_1(x),dh))=\pi_1^{\Omega,\Lambda}(x,\tilde{i}_1^*(dh))=\pi_1^{\Omega,\Lambda}(x,dh_1)=
dh_1(x),$$ where we only need to check the equality $(\tilde{i}_1^*(dh))=dh_1$, where $dh\in\Omega^1(X_1\cup_f X_2)$ and $dh_1\in\Omega^1(X_1)$. Indeed, let $p_1$ be a plot of $X_1$; then 
$(\tilde{i}_1^*(dh))(p_1)=dh(\tilde{i}_1\circ p_1)=d(h\circ\tilde{i}_1\circ p_1)$ by definition. Since $dh_1(p_1)=d(h_1\circ p_1)$ and $h\circ\tilde{i}_1=h_1$, we immediately obtain the desired conclusion. 
We have in fact obtained slightly more, namely, that the equalities stated hold on the entire domain of definition of $\tilde{\rho}_1^{\Lambda}$, that is, we have
$$\tilde{\rho}_1^{\Lambda}(dh(\tilde{i}_1(x)))=dh_1(x)\mbox{ for all }x\in X_1.$$ Observe furthermore that the case of $i_2(x)$ for $x\in X_2$ is treated in exactly the same manner, so we have that 
$$\tilde{\rho}_2^{\Lambda}(dh(i_2(x)))=dh_2(x)\mbox{ for all }x\in X_2.$$ 

Let us now confront the two sides of the equality in the Leibniz rule, considering 
$$\nabla^{\cup}(hs)=\nabla^{\cup}\left((h_1\cup_f h_2)(s_1\cup_{(f,\tilde{f})}s_2)\right).$$ Let $x\in X_1\cup_f X_2$; between the cases $x\in i_1(X_1\setminus Y)$ and $x\in i_2(X_2\setminus f(Y))$ it suffices 
to consider one, as they are symmetric. Let us consider $x\in i_1(X_1\setminus Y)$:
\begin{flushleft}
$(\nabla^{\cup}(hs))(x)=\left((\tilde{\rho}_1^{\Lambda})^{-1}\otimes j_1\right)\left((\nabla^1(h_1s_1))(i_1^{-1}(x))\right)=
\left((\tilde{\rho}_1^{\Lambda})^{-1}\otimes j_1\right)\left((dh_1\otimes s_1+h_1\nabla^1s_1)(i_1^{-1}(x))\right)=$
\end{flushleft}
\begin{flushright}
$=\left((\tilde{\rho}_1^{\Lambda})^{-1}(dh_1)\otimes s\right)(x)+\left((\tilde{\rho}_1^{\Lambda})^{-1}\otimes j_1\right)\left((h_1\nabla^1s_1)(i_1^{-1}(x))\right)=(dh\otimes s)(x)+(h\nabla^{\cup}s)(x)$, 
\end{flushright} as wanted. It thus remains to consider a point in $i_2(f(Y))$.

Let $x\in i_2(f(Y))$. Consider
$$(\nabla^{\cup}(hs))(x)=$$
\begin{flushleft}
$=((\tilde{\rho}_1^{\Lambda}\oplus\tilde{\rho}_2^{\Lambda})^{-1}\otimes\mbox{Id}_{V_1\cup_{\tilde{f}V_2}})
(\left(\mbox{Id}_{\Lambda^1(X_1)}\otimes(j_2\circ\tilde{f})\right)\left((\nabla^1(h_1s_1))(\tilde{i}_1^{-1}(x))\right)\oplus$
\end{flushleft}
\begin{flushright}
$\oplus\left(\mbox{Id}_{\Lambda^1(X_2)}\otimes j_2\right)\left((\nabla^2(h_2s_2))(i_2^{-1}(x))\right))=$
\end{flushright}
\begin{flushleft}
$=((\tilde{\rho}_1^{\Lambda}\oplus\tilde{\rho}_2^{\Lambda})^{-1}\otimes\mbox{Id}_{V_1\cup_{\tilde{f}V_2}})
(\left(\mbox{Id}_{\Lambda^1(X_1)}\otimes(j_2\circ\tilde{f})\right)\left((dh_1\otimes s_1+h_1\nabla^1s_1)(\tilde{i}_1^{-1}(x))\right)\oplus$
\end{flushleft}
\begin{flushright}
$\oplus\left(\mbox{Id}_{\Lambda^1(X_2)}\otimes j_2\right)\left((dh_2\otimes s_2+h_2\nabla^2s_2)(i_2^{-1}(x))\right))=$
\end{flushright} 
$$=\left(\tilde{\rho}_1^{\Lambda}\oplus\tilde{\rho}_2^{\Lambda}\right)^{-1}(dh_1+dh_2))(x)\otimes s(x)+(h\nabla^{\cup}s)(x).$$ It thus remains to check that 
for any $x\in i_2(f(Y))$ we have
$$dh(x)=\left((\tilde{\rho}_1^{\Lambda}\oplus\tilde{\rho}_2^{\Lambda})^{-1}(dh_1+dh_2)\right)(x).$$ This is equivalent to 
$$\tilde{\rho}_1^{\Lambda}(dh(x))=dh_1(f^{-1}(i_2^{-1}(x)))\,\,\mbox{ and }\tilde{\rho}_2^{\Lambda}(dh(x))=dh_2(i_2^{-1}(x)),$$ and this has already been established above, which completes the proof. 
\end{proof}

From the proof just finished, we can extract the following description of the differential of a function $h\in C^{\infty}(X_1\cup_f X_2,\matR)$ in terms of the differentials of its factors.

\begin{cor}\label{dh:under:gluing:cor}
The following is true:
$$(d(h_1\cup_f h_2))(x)=\left\{\begin{array}{cl}
(\tilde{\rho}_1^{\Lambda})^{-1}(dh_1(i_1^{-1}(x))) & \mbox{if }x\in i_1(X_1\setminus Y), \\ 
(\tilde{\rho}_1^{\Lambda}\oplus\tilde{\rho}_2^{\Lambda})^{-1}(dh_1(\tilde{i}_1^{-1}(x))\oplus dh_2(i_2^{-1}(x))) & \mbox{if }x\in i_2(f(Y)), \\ 
(\tilde{\rho}_2^{\Lambda})^{-1}(dh_2(i_2^{-1}(x))) & \mbox{if }x\in i_2(X_2\setminus f(Y)). \end{array}\right.$$
\end{cor}

\begin{prop}\label{nabla:is:smooth:as:operator:prop}
The operator $\nabla^{\cup}$ is smooth as a map
$$\nabla^{\cup}:C^{\infty}(X_1\cup_f X_2,V_1\cup_{\tilde{f}}V_2)\to C^{\infty}(X_1\cup_f X_2,\Lambda^1(X_1\cup_f X_2)\otimes(V_1\cup_{\tilde{f}}V_2))$$ for the usual functional diffeologies on the two 
spaces.
\end{prop}

\begin{proof}
Let $p:U\to C^{\infty}(X_1\cup_f X_2,V_1\cup_{\tilde{f}}V_2)$ be a plot of $C^{\infty}(X_1\cup_f X_2,V_1\cup_{\tilde{f}}V_2)$; we need to check that $u\mapsto\nabla^\cup(p(u))$ is a plot of 
$C^{\infty}(X_1\cup_f X_2,\Lambda^1(X_1\cup_f X_2)\otimes(V_1\cup_{\tilde{f}}V_2))$. Since the latter has functional diffeology, we need to check that for any plot $q:U'\to X_1\cup_f X_2$ of the base space 
$X_1\cup_f X_2$, the evaluation map 
$$\epsilon_{p,q}:(u,u')\mapsto\left(\nabla^{\cup}(p(u))\right)(q(u'))$$ is a plot of $\Lambda^1(X_1\cup_f X_2)\otimes(V_1\cup_{\tilde{f}}V_2)$. As usual, it suffices to assume that both $U$ and $U'$ are 
connected. This in particular implies that $q$ lifts to either a plot $q_1$ of $X_1$ or a plot $q_2$ of $X_2$.

Assume first that $q$ lifts to $q_1$ and consider $\epsilon_{p,q}(u,u')$ for an arbitrary point $(u,u')\in U\times U'$. Recall as a preliminary consideration that, since by assumption $p$ is a plot of 
$C^{\infty}(X_1\cup_f X_2,V_1\cup_{\tilde{f}}V_2)$, the following map (the corresponding version of the evaluation map) is smooth:
$$(u,u')\mapsto(p(u))(\tilde{i}_1^{-1}(q(u'))) =(p(u))(q_1(u')).$$ Since for each $u\in U$ the image $p(u)$ is a smooth section of $V_1\cup_{\tilde{f}}V_2$, it decomposes as $p(u)=p(u)_1\cup_{(f,\tilde{f})}p(u)_2$, 
where $p(u)_1=\tilde{j}_1^{-1}\circ p(u)\circ\tilde{i}_1\in C^{\infty}(X_1,V_1)$ and $p(u)_2=j_2^{-1}\circ p(u)\circ i_2\in C^{\infty}(X_2,V_2)$. 

We have by construction
$$\epsilon_{p,q}(u,u')=\left\{\begin{array}{cl}
\left((\tilde{\rho}_1^{\Lambda})^{-1}\otimes j_1\right)\left((\nabla^1p(u)_1)(q_1(u'))\right) & \mbox{if }q_1(u')\in X_1\setminus Y, \\
((\tilde{\rho}_1^{\Lambda}\oplus\tilde{\rho}_2^{\Lambda})^{-1}\otimes\mbox{Id}_{V_1\cup_{\tilde{f}V_2}})
(\left(\mbox{Id}_{\Lambda^1(X_1)}\otimes(j_2\circ\tilde{f})\right)\left((\nabla^1p(u)_1)(q_1(u'))\right)\oplus & \\
\oplus\left(\mbox{Id}_{\Lambda^1(X_2)}\otimes j_2\right)\left((\nabla^2p(u)_2)(f(q_1(u')))\right)) & \mbox{if }q_1(u')\in Y.
\end{array}\right.$$ Recall that by Theorem \ref{subduction:product-comp:onto:glued:thm} the two assignments $u\mapsto p(u)_1$ and $u\mapsto p(u)_2$ defined shortly above are plots of 
$C^{\infty}(X_1,V_1)$ and $C^{\infty}(X_2,V_2)$ respectively. In particular, by assumption we have that $(u,u')\mapsto(\nabla^1p(u)_1)(q_1(u'))$ is smooth as a map into $\Lambda^1(X_1)\otimes V_1$ onto 
the set of all pairs $(u,u')$ such that the expression $(\nabla^1p(u)_1)(q_1(u'))$ makes sense.

On the other hand, we cannot immediately make a similar claim regarding $(\nabla^2p(u)_2)(f(q_1(u')))$; indeed, $f$ is smooth for the subset diffeology on $Y$, to which $q_1|_{q_1^{-1}(Y)}$ might not belong. 
To draw the desired conclusion nonetheless, consider a plot $h:U''\to\mbox{Domain}(\epsilon_{p,q})\subset U\times U'$, which is just an ordinary smooth function. We need to show that $\epsilon_{p,q}\circ h$ 
is a plot of $\Lambda^1(X_1\cup_f X_2)\otimes(V_1\cup_{\tilde{f}}V_2)$. 

To do so, present $h$ as a pair of smooth functions $(h_U,h_{U'})$, where $h_U$ is the composition of $h$ with the projection of its range on $U$ and likewise $h_{U'}$ is its composition with the projection 
on $U'$. The composition $\epsilon_{p,q}\circ h$ is then the evaluation of $p\circ h_U$ on $q\circ h_{U'}$. It then remains to notice that $q\circ h_{U'}$ also lifts to a plot $(q\circ h_{U'})_1$ of $X_1$, and 
this lift is a plot for the subset diffeology on $Y$. Thus,
\begin{flushleft}
$(\epsilon_{p,q}\circ h)(u'')=((\tilde{\rho}_1^{\Lambda}\oplus\tilde{\rho}_2^{\Lambda})^{-1}\otimes\mbox{Id}_{V_1\cup_{\tilde{f}V_2}})
(\left(\mbox{Id}_{\Lambda^1(X_1)}\otimes(j_2\circ\tilde{f})\right)\left((\nabla^1p(h_U(u''))_1)((q_1\circ h_{U'})(u'')))\right)\oplus$
\end{flushleft}
\begin{flushright}
$\oplus\left(\mbox{Id}_{\Lambda^1(X_2)}\otimes j_2\right)\left((\nabla^2p(h_U(u''))_2)(f((q_1\circ h_{U'})(u'')))\right))$,
\end{flushright} and in particular $u''\mapsto(\nabla^2p(h_U(u''))_2)(f((q_1\circ h_{U'})(u'')))$ is now smooth by assumption on $\nabla^2$. We can therefore conclude that $\epsilon_{p,q}\circ h$ is indeed 
smooth, which completes the consideration of the case when $q$ lifts to a plot of $X_1$.

The treatment of the case when $q$ lifts to a plot $q_2$ of $X_2$ is completely analogous, so we omit it.
\end{proof}

This sequence of statements now trivially yields the following.

\begin{cor}\label{nabla:is:connection:cor}
The operator $\nabla^{\cup}$ is a connection on $V_1\cup_{\tilde{f}}V_2$.
\end{cor}

\begin{proof}
This is a consequence of Lemmas \ref{nabla-s:is:map:into:lambda-glued:lem}, \ref{nabla:on:section:of:glued:is:smooth:lem}, \ref{nabla:satisfies:leibnitz:lem}, and of Proposition 
\ref{nabla:is:smooth:as:operator:prop}.
\end{proof}

\begin{thm}
Let $\pi_1:V_1\to X_1$ and $\pi_2:V_2\to X_2$ be two diffeological vector pseudo-bundles, let $(\tilde{f},f)$ be a gluing between them such that both $\tilde{f}$ and $f$ are differmorphisms of their domains 
with their images, and $f$ is furthermore such that $\calD_1^{\Omega}=\calD_2^{\Omega}$, and let $\nabla^1$ and $\nabla^2$ be compatible connections on $V_1$ and $V_2$ respectively. Then 
$V_1\cup_{\tilde{f}}V_2$ can be endowed with a connection, that over $i_1(X_1\setminus Y)$ is naturally equivalent to $\nabla^1$ and over $i_2(X_2\setminus f(Y))$, to $\nabla^2$.
\end{thm}

\begin{proof}
This is the content of Corollary \ref{nabla:is:connection:cor}; the operator $\nabla^{\cup}$ corresponding to $\nabla^1$ and $\nabla^2$ is a connection and satisfies the claim of the theorem.
\end{proof}

\subsection{Compatibility of the induced connection $\nabla^{\cup}$ with the induced pseudo-metric $\tilde{g}$}

Assume now that the two pseudo-bundles $\pi_1:V_1\to X_1$ and $\pi_2:V_2\to X_2$ are endowed with pseudo-metrics $g_1$ and $g_2$ respectively, and that these pseudo-metrics are compatible with 
the gluing along $(\tilde{f},f)$:
$$g_1(y)(\cdot,\cdot)=g_2(f(y))(\tilde{f}(\cdot),\tilde{f}(\cdot))\,\,\mbox{ for all }\,\,y\in Y.$$ Let $\nabla^1$ be a connection on $V_1$ compatible with $g_1$, and let $\nabla^2$ be a connection on $V_2$ 
compatible with $g_2$. We can then consider the pseudo-metric $\tilde{g}$ on $V_1\cup_{\tilde{f}}V_2$ obtained by gluing together $g_1$ and $g_2$, and the connection
$\nabla^{\cup}$ on it. We wish to show that $\nabla^{\cup}$ is compatible with $\tilde{g}$.

Recall first that $\tilde{g}$ is defined by
$$\tilde{g}(x)(\cdot,\cdot)=\left\{\begin{array}{ll} g_1(i_1^{-1}(x))(j_1^{-1}(\cdot),j_1^{-1}(\cdot)) & \mbox{if }x\in i_1(X_1\setminus Y) \\
g_2(i_2^{-1}(x))(j_2^{-1}(\cdot),j_2^{-1}(\cdot)) & \mbox{if }x\in i_2(X_2).
\end{array}\right.$$ Thus, at least over $i_1(X_1\setminus Y)$ and $i_2(X_2\setminus f(Y))$ the compatibility would follow from the assumption on $\nabla^1$ and $\nabla^2$ respectively.

\begin{thm}\label{induced:connection:is:compatible:with:induced:pseudo-metric:thm}
Let $\pi_1:V_1\to X_1$ and $\pi_2:V_2\to X_2$ be two diffeological vector pseudo-bundles, let $(\tilde{f},f)$ be a gluing between them such that both $\tilde{f}$ and $f$ are differmorphisms of their domains 
with their images, and $\calD_1^{\Omega}=\calD_2^{\Omega}$ is satisfied, and let $\nabla^1$ and $\nabla^2$ be compatible connections on $V_1$ and $V_2$ respectively. Suppose furthermore that 
$V_1$ and $V_2$ are endowed with pseudo-metrics $g_1$ and $g_2$ that are compatible with the gluing along $f$ and $\tilde{f}$. Assume finally that $\nabla^1$ is compatible with $g_1$, and $\nabla^2$ 
is compatible with $g_2$. Then the induced connection $\nabla^{\cup}$ on $V_1\cup_{\tilde{f}}V_2$ is compatible with the induced pseudo-metric $\tilde{g}$.
\end{thm}

\begin{proof}
Let $s,t\in C^{\infty}(X_1\cup_f X_2,V_1\cup_{\tilde{f}}V_2)$ be two sections. We need to prove the following:
$$d(\tilde{g}(s,t))=\tilde{g}(\nabla^{\cup}s,t)+\tilde{g}(s,\nabla^{\cup}t).$$ Consider the usual splittings of $s$ and $t$ as $s=s_1\cup_{(f,\tilde{f})}s_2$ and $t=t_1\cup_{(f,\tilde{f})}t_2$, where 
$s_1,t_1\in C^{\infty}(X_1,V_1)$ and $s_2,t_2\in C^{\infty}(X_2,V_2)$. For these splittings, we have by assumption
$$d(g_1(s_1,t_1))=g_1(\nabla^1s_1,t_1)+g_1(s_1,\nabla^1t_1)\,\,\,\mbox{ and }\,\,\,d(g_2(s_2,t_2))=g_2(\nabla^2s_2,t_2)+g_2(s_2,\nabla^2t_2).$$ Since the differential is involved, and by 
Corollary \ref{dh:under:gluing:cor}, we need to consider three cases, those of a point in $i_1(X_1\setminus Y)$, a point in $i_2(f(Y))$, and one in $i_2(X_2\setminus f(Y))$, although the definition of $\tilde{g}$ 
only has two parts. We also express the function 
$$h_{\tilde{g},s,t}:X_1\cup_f X_2\ni x\mapsto\tilde{g}(x)(s(x),t(x))\in\matR$$ as the result of gluing of the following two functions:
$$h_{g_1,s_1,t_1}:X_1\ni x_1\mapsto g_1(x_1)(s_1(x_1),t_1(x_1))\,\,\mbox{ and }\,\,h_{g_2,s_2,t_2}:X_2\ni x_2\mapsto g_2(x_2)(s_2(x_2),t_2(x_2)).$$ It is then trivial to check that the gluing of these 
two functions along $f$ is well-defined (that is, that they are compatible with $f$, which in turn follows from the compatibility of $g_1$ with $g_2$), and that 
$$h_{\tilde{g},s,t}=h_{g_1,s_1,t_1}\cup_f h_{g_2,s_2,t_2}.$$

Consider now the first case, $x\in i_1(X_1\setminus Y)$. Then by Corollary \ref{dh:under:gluing:cor} and the observation just made
$$d(\tilde{g}(s,t))(x)=(\tilde{\rho}_1^{\Lambda})^{-1}(d(g_1(s_1,t_1)(i_1^{-1}(x))))=(\tilde{\rho}_1^{\Lambda})^{-1}((g_1(\nabla^1s_1,t_1)+g_1(s_1,\nabla^1t_1))(i_1^{-1}(x))).$$ It is thus sufficient to show that 
at a point $x\in i_1(X_1\setminus Y)$ we have 
$$(\tilde{\rho}_1^{\Lambda})^{-1}((g_1(\nabla^1s_1,t_1)(i_1^{-1}(x)))=\tilde{g}(\nabla^{\cup}s,t)(x),$$ and this is a direct consequence of the construction of $\nabla^{\cup}$. The completely analogous reasoning 
holds also in the case of $x\in i_2(X_2\setminus f(Y))$.

It thus remains to consider the case of $x\in i_2(f(Y))$. For such an $x$ we have, first of all,
$$d(\tilde{g}(s,t))(x)=
(\tilde{\rho}_1^{\Lambda}\oplus\tilde{\rho}_2^{\Lambda})^{-1}
(\mbox{Id}_{\Lambda^1(X_1)}\oplus\mbox{Id}_{\Lambda^1(X_2)})((dh_{g_1,s_1,t_1})(\tilde{i}_1^{-1}(x))+
(dh_{g_2,s_2,t_2})(i_2^{-1}(x))).$$ As follows from the assumptions on $\nabla^i$, and the linearity properties, what we now need to check is that for any $x\in i_2(f(Y))$ we have
$$(\tilde{\rho}_1^{\Lambda}\oplus\tilde{\rho}_2^{\Lambda})^{-1}
(\mbox{Id}_{\Lambda^1(X_1)}\oplus\mbox{Id}_{\Lambda^1(X_2)})(g_1(\nabla^1s_1,t_1)(\tilde{i}_1^{-1}(x))+g_2(\nabla^2s_2,t_2)(i_2^{-1}(x)))=
\tilde{g}(\nabla^{\cup}s,t).$$ This is also explicit from the construction of $\nabla^{\cup}$, which completes the proof.
\end{proof}

\begin{rem}
One might also consider the potential interplay between the two compatibility notions, one for connections and the other for pseudo-metrics, along the lines of whether one would imply the other (likely, the 
former, the latter). The proof just given indeed strongly suggests this possibility, at least as long as there are local bases. However, since in general diffeological pseudo-bundles do not have to have them, we 
do not follow through on this issue.
\end{rem}

\vspace{1cm}

\noindent University of Pisa \\
Department of Mathematics \\
Via F. Buonarroti 1C\\
56127 PISA -- Italy\\
\ \\
ekaterina.pervova@unipi.it\\

\end{document}